\documentclass[11pt]{amsart}
\usepackage{amsmath}
\usepackage{amsbsy}
\usepackage{amssymb}
\usepackage{amscd}
\usepackage{epsfig}
\usepackage[mathscr]{eucal}
\usepackage{mathrsfs}
\usepackage{url}
\usepackage{mathtools}
\newcounter{foo}

\def\rs#1{\expandafter{\romannumeral#1}}%

\newcommand{\Z}{\mathbb{Z}}
\newcommand{\Q}{\mathbb{Q}}
\newcommand{\R}{\mathbb{R}}

\newcommand{\C}{\mathbb{C}}
\newcommand{\h}{\mathcal{H}}

\newcommand{\dR}{\mathrm{dR}}
\newcommand{\GL}{\mathrm{GL}}
\newcommand{\SL}{\mathrm{SL}}

\newcommand{\Fil}{\operatorname{Fil}}

\newtheorem{thm}{Theorem}[section]
\newtheorem{lem}[thm]{Lemma}
\newtheorem{cor}[thm]{Corollary}
\newtheorem{prop}[thm]{Proposition}

\newtheorem{conj}[thm]{Conjecture}

\theoremstyle{definition}

\newtheorem{defn}[thm]{Definition}

\newtheorem{rem}[thm]{Remark}

\theoremstyle{remark}

\begin{document}
%\begin{frontmatter}
\title[\ Regulators for Rankin-Selberg products]{\ Regulators for Rankin-Selberg products of modular forms}
\author{Fran\c{c}ois Brunault and Masataka Chida}

\renewcommand{\thefootnote}{\fnsymbol{footnote}}
%\footnote[0]{Email:\ chida@math.kyoto-u.ac.jp}
%\footnote[0]{Address:\ Department of Mathematics, Kyoto University, Kitashirakawa Oiwake-cho, Kyoto, 606-8502,Japan}
%\renewcommand{\subjclassname}{%
% \textup{2000} Mathematics Subject Classification}
%%%%%%%%%%%%%%%%%%%%%%%%%%%%%%%%\subjclass[2000]{11F33,11F37,11F67}

%\subjclass{Primary 11F67 ; Secondary11F37. }
%%%%\begin{frontmatter}
\date{\today}
\address{\'ENS Lyon, UMPA, 46 all\'ee d'Italie, 69007 Lyon, France}
\email{francois.brunault@ens-lyon.fr}
\address{Department of Mathematics, Kyoto University, Kitashirakawa Oiwake-cho, Sakyo-ku, Kyoto, 606-8502, Japan}
\email{chida@math.kyoto-u.ac.jp}

%%%%\end{footmatter}
\begin{abstract}
We prove a weak version of Beilinson's conjecture for non-critical values of $L$-functions for the Rankin-Selberg product of two modular forms.
\end{abstract}
\maketitle
%\tableofcontents

\section*{Introduction}

In his fundamental paper \cite[\S 6]{Beilinson1985}, Beilinson introduced the so-called \emph{Beilinson-Flach elements} in the higher Chow group of a product of two modular curves and related their image under the regulator map to special values of Rankin $L$-series of the form $L(f \otimes g,2)$, where $f$ and $g$ are newforms of weight $2$, as predicted by his conjectures on special values of $L$-functions. These elements were later exploited by Flach \cite{Flach} to prove the finiteness of the Selmer group associated to the symmetric square of an elliptic curve. More recently, Bertolini, Darmon and Rotger \cite{BDR} established a $p$-adic analogue of Beilinson's result, while Lei, Loeffler and Zerbes \cite{LLZ} constructed a cyclotomic Euler system whose bottom layer are the Beilinson-Flach elements. These results have many important arithmetic applications (\cite{BDR2}, \cite{LLZ}).

Our aim in this paper is to define an analogue of the Beilinson-Flach elements in the motivic cohomology of a product of two Kuga-Sato varieties and to prove an analogue of Beilinson's formula for special values of Rankin $L$-series associated to newforms $f$ and $g$ of any weight $\geq 2$. More precisely, we prove the following theorem.

\begin{thm}
Let $f \in S_{k+2}(\Gamma_1(N_f),\chi_f)$ and $g \in S_{\ell+2}(\Gamma_1(N_g),\chi_g)$ be newforms with $k,\ell \geq 0$. Assume that the Dirichlet character $\chi$ modulo $N=\operatorname{lcm}(N_f,N_g)$ induced by $\chi_f \chi_g$ is non-trivial. Let $j$ be an integer satisfying $0 \leq j \leq \min \{ k,\ell\}$. Assume that the automorphic factor $R_{f,g,N}(j+1)$ defined in Section \ref{Rankin-Selberg} is non-zero (this holds for example if $\gcd(N_f,N_g)=1$ or if $k+\ell-2j \not\in \{0,1,2\}$). Then the weak version of Beilinson's conjecture for $L(f \otimes g, k+ \ell+2-j)$ holds.
\end{thm}

The range of critical values (in the sense of Deligne) for the Rankin-Selberg $L$-function $L(f \otimes g,s)$ is given by $\min \{ k,\ell \} +2 \leq s \leq \max \{ k,\ell\} +1$, so that our $L$-value $L(f \otimes g,k+\ell+2-j)$ is \emph{non-critical}. In fact, the integers $0 \leq j \leq \min \{k,\ell\}$ are precisely those at which the dual $L$-function $L(f^* \otimes g^*,s+1)$ vanishes \emph{at order 1}. 

We refer to Theorem \ref{regulator} for the explicit formula giving the regulator of our generalized Beilinson-Flach elements. In the weight $2$ case, an explicit version of Beilinson's formula for $L(f \otimes g,2)$, similar to Theorem \ref{regulator}, was proved by Baba and Sreekantan \cite{BabaSreekantan} and by Bertolini, Darmon and Rotger \cite{BDR}. In the higher weight case, a similar formula for the regulator of generalized Beilinson-Flach elements was proved by Scholl (unpublished) and recently by Kings, Loeffler and Zerbes \cite{KLZ}. As a difference with \cite{KLZ}, we work directly with the motivic cohomology of the Kuga-Sato varieties (instead of motivic cohomology with coefficients), and we prove that our generalized Beilinson-Flach elements extend to the boundary of the Kuga-Sato varieties (see Sections \ref{residues} and \ref{ext boundary}). Another interesting problem is the integrality of the generalized Beilinson-Flach elements. In the case $f$ and $g$ have weight $2$, Scholl proved that if $g$ is not a twist of $f$, then the Beilinson-Flach elements belong to the integral subspace of motivic cohomology \cite[Theorem 2.3.4]{Scholl integrality}. We do not investigate integrality in this article, but it would be interesting to do so using Scholl's techniques.

The plan of this article is as follows. In Section \ref{sec beilinson}, we recall the statement of Beilinson's conjecture for Grothendieck motives. In Section \ref{sec modular forms}, we recall some basic results about motives of modular forms and describe explicitly the Deligne cohomology group associated to the Rankin product of two modular forms. After recalling Beilinson's theory of the Eisenstein symbol (Section \ref{sec eisenstein}), we construct in Section \ref{construction} special elements $\Xi^{k,\ell,j}(\beta)$ in the motivic cohomology of the product of two Kuga-Sato varieties. After recalling standard facts about the Rankin-Selberg $L$-function (Section \ref{Rankin-Selberg}), we compute the regulator of our elements $\Xi^{k,\ell,j}(\beta)$ in Section \ref{computation}. We then show in Sections \ref{residues} and \ref{ext boundary}, using motivic techniques, that a suitable modification of the elements $\Xi^{k,\ell,j}(\beta)$ extends to the boundary of the Kuga-Sato varieties. Finally, we give in Section \ref{application} the application of our results to Beilinson's conjecture.

We would like to thank Fr\'ed\'eric D\'eglise and J\"org Wildeshaus for very stimulating discussions on the topic of motives.
We also thank Jan Nekov\'a\v{r} for helpful discussions and comments on our work.
The second author is partly supported by JSPS KAKENHI Grant Number 23740015
and the research grant of Hakubi project of Kyoto University.

\section{Beilinson's conjecture}\label{sec beilinson}

Let $X$ be a smooth projective variety over $\Q$.
For a non-negative integer $i$ and an integer $j$, let $H_{\mathcal{M}}^i(X,\Q (j))$ be the motivic
cohomology and $H_{\mathcal{D}}^i(X_\R,\R (j))$ be the Deligne cohomology.
Then one can define natural $\Q$-structures $\mathcal{B}_{i,j}$ and $\mathcal{D}_{i,j}$ in $\mathrm{det}_\R (H_{\mathcal{D}}^{i+1}(X_\R,\R (j)))$
(see Deninger-Scholl \cite[2.3.2]{Deninger-Scholl} or Nekov\'a\v{r} \cite[(2.2)]{Nekovar}).
Denote the integral part of motivic cohomology by
$H_{\mathcal{M}}^i(X,\Q (j))_{\Z}$.
Then Beilinson defined the regulator map
$$
r_{\mathcal{D}} :H_{\mathcal{M}}^{i+1}(X,\Q (j))\to H_{\mathcal{D}}^{i+1}(X_\R,\R (j))
$$
and formulated a conjecture for the special values of the $L$-function $L(h^i(X),s)$ as follows.
\begin{conj}[Beilinson \cite{Beilinson1985}]\label{BeilinsonConj}
Assume $j>(i+2)\slash 2$.
\begin{enumerate}
\item The map $r_{\mathcal{D}}\otimes \R : H_{\mathcal{M}}^{i+1}(X,\Q (j))_{\Z}\otimes \R 
\to H_{\mathcal{D}}^{i+1}(X_\R,\R (j))$ is an isomorphism.
\item We have $r_{\mathcal{D}}(\mathrm{det} H_{\mathcal{M}}^{i+1}(X,\Q (j))_{\Z})=L(h^i(X),j) \cdot \mathcal{D}_{i,j}=L^*(h^i(X),i+1-j) \cdot \mathcal{B}_{i,j}$,
where $L^*(h^i(X),m)$ is the leading term of the Taylor expansion of $L(h^i(X),s)$ at $s=m$.
\end{enumerate}
\end{conj}
In the case of the near-central point $j=(i+2)\slash 2$, we have the following modified conjecture. Let $N^{j-1}(X)=\mathrm{CH}^{j-1}(X)_{\mathrm{hom}} \otimes \Q$ be the group of $(j-1)$-codimensional cycles modulo homological equivalence. The cycle class map into de Rham cohomology defines an extended regulator map
$$
\hat{r}_{\mathcal{D}} : H_{\mathcal{M}}^{i+1}(X,\Q (j)) \oplus N^{j-1}(X) \to H_{\mathcal{D}}^{i+1}(X_\R,\R (j)).
$$

\begin{conj}[Beilinson \cite{Beilinson1985}]\label{BeilinsonConj near-central}
Assume $j=(i+2)\slash 2$.
\begin{enumerate}
\item (Tate's conjecture) We have $\mathrm{ord}_{s=j} L(h^i(X),s) = -\dim_\Q N^{j-1}(X)$.
\item The map $\hat{r}_{\mathcal{D}}\otimes \R : (H_{\mathcal{M}}^{i+1}(X,\Q (j))_\Z \oplus N^{j-1}(X)) \otimes \R \to H_{\mathcal{D}}^{i+1}(X_\R,\R (j))$ is an isomorphism.
\item We have $\hat{r}_{\mathcal{D}}(\mathrm{det} (H_{\mathcal{M}}^{i+1}(X,\Q (j))_{\Z} \oplus N^{j-1}(X)))=L^*(h^i(X),j) \cdot \mathcal{D}_{i,j}=L^*(h^i(X),j-1) \cdot \mathcal{B}_{i,j}$.
\end{enumerate}
\end{conj}

We will now formulate a version of Beilinson's conjectures for Grothendieck motives. Let $M=(X,p)$ be a Grothendieck motive over $\Q$ with coefficients in $L$,
where $X$ is a smooth projective variety over $\Q$ and $p$ is a projector
in $\mathrm{CH}^{\mathrm{dim}X}(X\times X)_{\mathrm{hom}}\otimes_\Q L$. We define the Deligne cohomology of $M$ by
$$
H^{\cdot}_{\mathcal{D}}(M,j)=p_* (H_{\mathcal{D}}^{\cdot}(X_\R,\R (j))\otimes L).
$$
Let us assume that $M$ is a direct factor of $h^i(X) \otimes L$. We have an $L$-function $L(M,s)=L(H^i(M),s)$ taking values in $L \otimes \C$. Moreover, there are natural $L$-structures $\mathcal{B}_{i,j}(M)$ and $\mathcal{D}_{i,j}(M)$ in
$\mathrm{det}_{L \otimes \R}H^{i+1}_{\mathcal{D}}(M,j)$. We define Beilinson's regulator as
$$
r_{\mathcal{D}}:H_{\mathcal{M}}^{i+1}(X,\Q (j))\otimes L \to H^{i+1}_{\mathcal{D}}(X_\R,\R (j))\otimes L
\to H^{i+1}_{\mathcal{D}}(M,j),
$$
where the last map is the projection induced by $p_*$. Similarly, we define an extended regulator map $\hat{r}_{\mathcal{D}}$ in the case $j=(i+2)/2$. Assume Conjecture \ref{BeilinsonConj} (1) for $X$. Then Beilinson's conjecture for $L(M,s)$ can be
formulated as follows.
\begin{conj}\label{motive}
\begin{enumerate}
\item If $j> (i+2)/2$, then $p_*(r_{\mathcal{D}}(\mathrm{det}_L H_{\mathcal{M}}^{i+1}(X,\Q (j))_{\Z}\otimes L))=L(M,j) \cdot \mathcal{D}_{i,j}(M)$.
\item If $j=(i+2)/2$, then $p_*(\hat{r}_{\mathcal{D}}(\mathrm{det}_L (H_{\mathcal{M}}^{i+1}(X,\Q (j))_{\Z} \oplus N^{j-1}(X))\otimes L))=L^*(M,j) \cdot \mathcal{D}_{i,j}(M)$.
\end{enumerate}
\end{conj}

Since Conjecture \ref{BeilinsonConj}(1) is in general out of reach, we formulate a weak version of Conjecture \ref{motive} as follows.

\begin{conj}[Weak version] \label{Conj Beilinson weak} Assume $M$ is a direct factor of $h^i(X) \otimes L$. Let $M^\vee = (X, \prescript{t}{}{p}, i)$ be the dual motive of $M$.
\begin{enumerate}
\item If $j>(i+2)\slash 2$, then there exists a subspace $V$ of $H_{\mathcal{M}}^{i+1}(X,\Q (j))$ such that $p_*(r_{\mathcal{D}}(V \otimes L))$ is an $L$-structure of $H_{\mathcal{D}}^{i+1}(M,j)$ and
$$
\operatorname{det} p_*(r_{\mathcal{D}}(V \otimes L)) = L(M,j) \cdot \mathcal{D}_{i,j}(M)=L^*(M^\vee,1-j) \cdot \mathcal{B}_{i,j}(M).
$$
\item If $j=(i+2)\slash 2$, then there exists a subspace $V$ of $H_{\mathcal{M}}^{i+1}(X,\Q (j)) \oplus N^{j-1}(X)$ such that $p_*(\hat{r}_{\mathcal{D}}(V \otimes L))$ is an $L$-structure of $H_{\mathcal{D}}^{i+1}(M,j)$ and
$$
\operatorname{det} p_*(\hat{r}_{\mathcal{D}}(V \otimes L)) = L^*(M,j) \cdot \mathcal{D}_{i,j}(M)=L^*(M^\vee,1-j) \cdot \mathcal{B}_{i,j}(M).
$$
\end{enumerate}
\end{conj}

\begin{rem}
We could have required a stronger property in Conjecture \ref{Conj Beilinson weak}, namely that $V$ is a subspace of $H_{\mathcal{M}}^{i+1}(X,\Q (j))_\Z$. But since we don't consider the problem of integrality of elements of motivic cohomology in this paper, we leave Conjecture \ref{Conj Beilinson weak} as it is.
\end{rem}

\section{Motives associated to modular forms}\label{sec modular forms}

Let us recall some basic properties of motives associated to modular forms.
Let $Y=Y(N)$ be the modular curve with full level $N$-structure defined over $\Q$
and $j:Y\hookrightarrow X=X(N)$ the smooth compactification.
Let $\pi : E \to Y $ be the universal elliptic curve over $Y$ and
$\bar{\pi}:\overline{E}\to X$ be the universal generalized elliptic curve.
Then $\overline{E}$ is smooth and proper over $\Q$.
For a non-negative integer $k$, denote the $k$-fold fiber product of $E$ over $Y$ by $E^k$
and the $k$-fold fiber product of $\overline{E}$ over $X$ by $\overline{E}^k$.
Let $\hat{E}^k$ be the N\'eron model of $E^k$ over $X$ and $\hat{E}^{k,*}$ the connected component.
If $k\geq 2$, then $\overline{E}^k$ is singular.
Let $\overline{\overline{E}}^k\to\overline{E}^k$ be the canonical desingularization constructed by Deligne.
%When the level needs to be specified, we write $E_N$, $\overline{E}_N$... for the Kuga-Sato varieties.

Let $f\in S_{k+2}(\Gamma_1(N), \chi)^{\mathrm{new}}$ ($k \geq 0$) be a normalized eigenform. Let $K_f \subset \C$ be the number field generated by the Fourier coefficients of $f$. Let $M(f)$ be the Grothendieck motive associated to $f$ \cite{Scholl}. It is a motive of rank $2$ defined over $\Q$ with coefficients in $K_f$. The motive $M(f)$ is a direct factor of $h^{k+1}(\overline{\overline{E}}^k) \otimes K_f$. By Grothendieck's theorem, we have an isomorphism $H^{k+1}_B(M(f)) \otimes \C \cong H^{k+1}_{\mathrm{dR}}(M(f)) \otimes \C$ between Betti and de Rham cohomology. The $K_f \otimes \C$-module $H^{k+1}_{\mathrm{dR}}(M(f)) \otimes \C$ is free of rank $2$, with basis $\{ \omega_f, \overline{\omega_{f^*}}\}$, where
$$
\omega_f = (2\pi i)^{k+1} f(\tau) d\tau \wedge dz_1 \wedge \cdots \wedge dz_k. 
$$
We denote $\omega'_f=G(\chi)^{-1}\omega_f$,
%For a Dirichlet character $\chi$, let $f_\chi$ be the conductor of $\chi$.
%%Put $\chi (-1)=(-1)^{\eta}$ with $\eta \in \{ 0,1 \}$.
%Then we define
%the Gauss sum $G(\chi )$ by 
%$\varepsilon$-factor by
where
$$
G (\chi)=%i^{\eta}
%f_{\chi}^s
\sum_{u=1}^{N_{\chi}}\chi (u) e^{2\pi i u\slash N_{\chi}}
$$
is the Gauss sum of the Dirichlet character $\chi$
and $N_{\chi}$ is the conductor of $\chi$.
By \cite[Lemma 6.1.1]{KLZ}, we have
$$
\Fil^i H^{k+1}_{\mathrm{dR}}(M(f)) = \begin{cases} H^{k+1}_{\mathrm{dR}}(M(f)) & \textrm{if } i \leq 0,\\
K_f \cdot \omega'_f & \textrm{if } 1 \leq i \leq k+1,\\
0 & \textrm{if } i \geq k+2.
\end{cases}
$$
By Poincar\'e duality, we have a perfect pairing of $K_f$-vector spaces
%$$
%H^{k+1}(E^k(\C),\Q(k+1)) \times H^{k+1}_c(E^k(\C),\Q) \to \Q.
%$$
%This induces a perfect pairing of $K_f$-vector spaces
$$
H^{k+1}_B(M(f^*)(k+1)) \times H^{k+1}_B(M(f)) \to K_f.
$$
Now, let $f\in S_{k+2}(\Gamma_1(N_f), \chi_f)^{\mathrm{new}}$, $g \in S_{\ell+2}(\Gamma_1(N_g), \chi_g)^{\mathrm{new}}$ ($\ell \geq k \geq 0$) be normalized eigenforms. We consider the Grothendieck motive
$$
M(f \otimes g) := M(f) \otimes M(g).
$$
This motive has coefficients in $K_{f,g}:=K_f K_g$ and is a direct factor of
$$
h^{k+1}(\overline{\overline{E}}_{N_f}^k) \otimes h^{\ell+1}(\overline{\overline{E}}_{N_g}^\ell) \otimes K_{f,g} \subset h^{k+\ell+2}(\overline{\overline{E}}_{N_f}^k \times \overline{\overline{E}}_{N_g}^\ell) \otimes  K_{f,g}.
$$
Let $j$ be an integer such that $0 \leq j \leq k$ and put $n=k+\ell+2-j$. The Deligne cohomology of $M(f \otimes g)(n)$ can be expressed as follows. The de Rham realization
$$
H^{k+\ell+2}_{\mathrm{dR}}(M(f \otimes g)) = H^{k+1}_{\mathrm{dR}}(M(f)) \otimes H^{\ell+1}_{\mathrm{dR}}(M(g))
$$
has dimension $4$ over $K_{f,g}$. Moreover $\Fil^n H^{k+\ell+2}_{\mathrm{dR}}(M(f \otimes g))$ is the $K_{f,g}$-line generated by $\omega'_f \otimes \omega'_g$. Then we have an exact sequence
\begin{equation}
%\label{HD 1} 0 \to H^{k+\ell+2}_B(M(f \otimes g)(n))^+ \otimes \R \to (H^{k+\ell+2}_{\mathrm{dR}}(M(f \otimes g))/\Fil^n) \otimes \R \to H^{k+\ell+3}_{\mathcal{D}}(M(f \otimes g)(n)) \to 0,\\
\label{HD 2} 0 \to \Fil^n H^{k+\ell+2}_{\mathrm{dR}}(M(f \otimes g)) \otimes \R \to H^{k+\ell+2}_B(M(f \otimes g)(n-1))^+ \otimes \R \to H^{k+\ell+3}_{\mathcal{D}}(M(f \otimes g)(n)) \to 0.
\end{equation}
In particular $H^{k+\ell+3}_{\mathcal{D}}(M(f \otimes g)(n))$ is a free $K_{f,g} \otimes \R$-module of rank 1.

The exact sequence (\ref{HD 2}) induces a $K_{f,g}$-rational structure on $H^{k+\ell+3}_{\mathcal{D}}(M(f \otimes g)(n))$. Let us make explicit a generator of this rational structure. Let $e_f^{\pm}$ be a $K_f$-basis of $H^{k+1}_B(M(f))^{\pm}$, and let $e_g^{\pm}$ be a $K_g$-basis of $H^{\ell+1}_B(M(g))^{\pm}$. Under the comparison isomorphism $H^{k+1}_{B}(M(f)) \otimes \C \cong H^{k+1}_{\dR}(M(f)) \otimes \C$, we have $\omega_f = \alpha_f^+ e_f^+ + \alpha_f^- e_f^-$ for some $\alpha_f^+,\alpha_f^- \in \C$. Note that $\alpha_f^+ \in \R$ and $\alpha_f^- \in i\R$. Similarly, let $\omega_g = \alpha_g^+ e_g^+ + \alpha_g^- e_g^-$. The $K_{f,g}$-vector space $H^{k+\ell+2}_B(M(f \otimes g)(n-1))^+$ admits as a $K_{f,g}$-basis $(e_1,e_2)$ where
\begin{align*}
e_1 & = e_f^+ \otimes e_g^{(-1)^{n+1}} \otimes (2\pi i)^{n-1},\\
e_2 & = e_f^- \otimes e_g^{(-1)^n} \otimes (2\pi i)^{n-1}.
\end{align*}
The image of $\omega'_f \otimes \omega'_g$ in $H^{k+\ell+2}_B(M(f \otimes g)(n-1))^+ \otimes \R$ under (\ref{HD 2}) is given by
\begin{align*}
\pi(\omega'_f \otimes \omega'_g) & = G(\chi_f)^{-1}G(\chi_g)^{-1}\alpha_f^- \alpha_g^{(-1)^n} e_f^- \otimes e_g^{(-1)^n} + G(\chi_f)^{-1}G(\chi_g)^{-1}\alpha_f^+ \alpha_g^{(-1)^{n+1}} e_f^+ \otimes e_g^{(-1)^{n+1}}\\
& = G(\chi_f)^{-1}G(\chi_g)^{-1}(2\pi i)^{1-n} (\alpha_f^+ \alpha_g^{(-1)^{n+1}} e_1 + \alpha_f^- \alpha_g^{(-1)^n} e_2).
\end{align*}
Thus a rational structure of $H^{k+\ell+3}_{\mathcal{D}}(M(f \otimes g)(n))$ is given by
\begin{equation*}
t := G(\chi_f)G(\chi_g)(2\pi i)^{n-1} (\alpha_f^- \alpha_g^{(-1)^n})^{-1} e_1.
\end{equation*}

Since $M(f \otimes g)(n-1)^\vee \cong M(f^* \otimes g^*)(j+1)$, we have a perfect pairing
\begin{align*}
H^{k+\ell+2}_B(M(f \otimes g)(n-1)) \times H^{k+\ell+2}_B(M(f^* \otimes g^*)(j+1)) \to K_{f,g}.
\end{align*}

Now, let us define a canonical element $\Omega \in H^{k+\ell+2}_B(M(f^* \otimes g^*)(j+1)) \otimes \C$, which we will use to pair with the regulator of our generalized Beilinson-Flach element.
Under the canonical isomorphism
$$
\phi_{\dR}:H^{k+\ell+2}_{\dR}(M(f^*)\otimes M(g^*)(j+1)) \overset{\cong}{\to} 
H^{k+\ell+2}_{\dR}(M(f) \otimes M(\chi_f)\otimes M(g)\otimes M(\chi_g)(j+1)),% \cong 
%H^{k+1}_{\dR}(M(f))\otimes H^0_{\dR}(M(\chi_f)) \otimes H^{\ell+1}_{\dR}(M(g))\otimes H^0_{\dR}(M(\chi_g)(j+1))$, 
$$
the element $G(\overline{\chi_f})^{-1}G(\overline{\chi_g})^{-1}\omega_{f^*}\otimes \omega_{g^*}$ corresponds to a $K_{f,g}^{\times}$-rational multiple of $\omega'_f \otimes \omega(\chi_f) \otimes \omega'_g \otimes \omega(\chi_g)$, where $\omega (\chi_f)$ is basis
of $H^0_{\dR}(M(\chi_f))$.

We recall the periods for motives associated to Dirichlet characters with coefficients in $E$.
%and denote $\varepsilon (\chi)=\varepsilon (\chi,0)$.
Let $M(\chi)$ be the motive associated to $\chi$ with coefficients in a number field $E$.
Then the period of the comparison isomorphism
$H^0_B(M(\chi ))=E(\chi) \to H^0_{\dR}(M(\chi))= G(\chi)\cdot E$
is given by $G(\chi)^{-1}$,
where $E(\chi)$ is the rank one $E$-vector space on which the Galois group
$\mathrm{Gal}(\Q (e^{2\pi i \slash N_{\chi}})\slash \Q )$ acts via $\chi$
and $G(\chi)\cdot E$ is the $E$-vector space generated by $G(\chi)$
(for details, see \cite[Section 6]{Deligne}).

Under the comparison isomorphism
$$\phi :
H^{k+\ell+2}_{B}(M(f) \otimes M(\chi_f)\otimes M(g)\otimes M(\chi_g)(j+1))
\otimes \C \xrightarrow{\cong}
H^{k+\ell+2}_{\dR}(M(f) \otimes M(\chi_f)\otimes M(g)\otimes M(\chi_g)(j+1)) \otimes \C,
$$
we have
%$\phi^{-1}(\omega_f \otimes \omega(\chi_f)\otimes \omega_g \times \omega(\chi_g))$
%is given by
\begin{equation*}
\phi^{-1}(\omega'_f \otimes \omega(\chi_f)\otimes \omega'_g \otimes \omega(\chi_g)) = 
(\alpha_f^+ e_f^+ + \alpha_f^- e_f^-)\otimes (\alpha_g^+ e_g^+ + \alpha_g^- e_g^-) \otimes  e(\chi_f) \otimes  e(\chi_g).
\end{equation*}
Let $e_f^{\pm,\vee}$ be a $K_f$-basis of $H^{k+1}_B(M(f)^\vee)^{\pm}$ with $\langle e_f^{\pm},e_f^{\pm,\vee}\rangle =1$, and let $e_g^{\pm,\vee}$ be a $K_g$-basis of $H^{\ell+1}_B(M(g)^\vee)^{\pm}$ with $\langle e_g^{\pm},e_g^{\pm,\vee}\rangle =1$. 
We have an isomorphism
$\phi_B(f) : H^{k+1}_B(M(f) \otimes M(\chi_f)) \xrightarrow{\cong} H^{k+1}_B(M(f)^\vee(-k-1))$
sending $e_f^\pm \otimes e(\chi_f)$ to a $K_{f,g}^{\times}$-rational multiple of
$(2\pi i)^{-k-1}e_f^{\mp,\vee}$.
Note that $(2\pi i)^{k+1}e(\chi_f) \in H_B^0(M(\chi_f)(k+1))^-$, since $\chi_f (-1) = (-1)^k$.
Therefore we have an isomorphism
$$
\phi_B :%(f\otimes g)(j+1):
H_B^{k+\ell+2}(M(f)\otimes M(g)\otimes M(\chi_f)\otimes M(\chi_g)(j+1))
\overset{\cong}{\to} H_B^{k+\ell+2}(M(f)^{\vee}\otimes M(g)^{\vee}(1-n))
$$
sending $(2 \pi i)^{j+1}e_f^{\pm}\otimes e_g^{\pm} \otimes e(\chi_f)\otimes e(\chi_g)$
to a rational multiple of $(2 \pi i)^{1-n}e_f^{\mp,\vee}\otimes e_g^{\mp,\vee}$.
Let us define
\begin{equation*}
\nu_f := \phi_B \circ \phi^{-1}
(\omega'_f \otimes \omega(\chi_f))=
 (2\pi i)^{-k-1}
(\alpha_f^+ e_f^{-,\vee}+\alpha_f^- e_f^{+,\vee})
= (2\pi i)^{-k-1}
(\alpha_f^+ e_f^{-,\vee}+\alpha_f^- e_f^{+,\vee})
\end{equation*}
%\begin{align*}
%\nu_f \otimes \nu_g &:= \phi_B \circ \phi^{-1}
%(\omega_f \otimes \omega(\chi_f)\otimes \omega_g \otimes \omega(\chi_g))\\ &=
%\varepsilon(\chi_f)^{-1} \varepsilon(\chi_g)^{-1} (2\pi i)^{-k-\ell-2}
%(\alpha_f^+ e_f^{-,\vee}+\alpha_f^- e_f^{+,\vee})\otimes (\alpha_g^+ e_g^{-,\vee}+\alpha_g^- e_g^{+,\vee}) \\
%& \in H^{k+\ell+2}_B(M(f)^{\vee}\otimes M(g)^\vee(1-n)) \otimes \C.
%\end{align*}
and
\begin{equation*}
\nu_g :=  (2\pi i)^{-\ell-1}(\alpha_g^+ e_g^{-,\vee}+\alpha_g^- e_g^{+,\vee})  .
\end{equation*}
%Let us define
Also we define
\begin{equation*}
%\nu_f \otimes \overline{\nu_{g^*}} :=
%\varepsilon(\chi_f)^{-1} \varepsilon(\chi_g)^{-1} (2\pi i)^{-k-\ell-2}
%(\alpha_f^+ e_f^{-,\vee}+\alpha_f^- e_f^{+,\vee})\otimes c_B(\alpha_g^+ e_g^{-,\vee}+\alpha_g^- e_g^{+,\vee})
%\in H^{k+\ell+2}_B(M(f)^{\vee}\otimes M(g)^\vee(1-n)) \otimes \C ,
\overline{\nu_{g^*}}= \overline{F}_{\infty}^*(\nu_g) =  (2\pi i)^{-\ell-1}(-\alpha_g^+ e_g^{-,\vee}+\alpha_g^- e_g^{+,\vee}),
% \in H^{\ell+1}_B(M(g)^\vee(-\ell-1)) \otimes \C.
\end{equation*}
where $\overline{F}_{\infty}^*$ is the involution defined in \cite[1.4]{Deligne}.
We define
\begin{equation*}
\Omega := G(\overline{\chi_f}) G(\overline{\chi_g}) \nu_f \otimes \overline{\nu_{g^*}} 
\in H^{k+\ell+2}_B(M(f \otimes g)^\vee (1-n)) \otimes \C = H^{k+\ell+2}_B(M(f^*\otimes g^*)(j+1))\otimes \C.
\end{equation*}
%Finally we may define
%\begin{equation*}
%\nu_{f\otimes g}:=
%\nu_f \otimes \overline{\nu_{g^*}} \in H^{k+\ell+2}_B(M(f^* \otimes g^*)(j+1)) \otimes \C.
%\end{equation*}
Since $\phi_B \circ \phi^{-1}\circ \phi_{\dR}(\omega'_{f^*}\otimes \overline{\omega'_g})=
\phi_B \circ \phi^{-1}\circ \phi_{\dR}(\omega'_{f^*}\otimes \overline{F}_{\infty}^*({\omega'_{g^*}}))$ is a $K_{f,g}^\times$-rational multiple of $\nu_{f}\otimes \overline{F}_{\infty}^* (\nu_{g})=\nu_{f}\otimes \overline{\nu_{g^*}}
$, it follows that $\phi_{\dR}^{-1}\circ \phi \circ \phi_B^{-1}(\Omega)$ is a $K_{f,g}^\times$-rational multiple of
$$G(\overline{\chi_f})G(\overline{\chi_g})\omega'_{f^*}\otimes \overline{\omega'_g}
=\omega_{f^*}
\otimes \overline{\omega_g} \in H^{k+\ell+2}_{\dR}(M(f^*\otimes g^*)(j+1))\otimes \mathbb{C}.$$
%where $c_{\dR}$ is the complex conjugation acting on de Rham cohomology.

\begin{lem}\label{pairing Omega}
The map
$$
\langle \, \cdot \, ,\Omega \rangle : 
H^{k+\ell+2}_B(M(f \otimes g)(n-1))^+ \otimes \R \to K_{f,g} \otimes \C
$$
factors through $H^{k+\ell+3}_{\mathcal{D}}(M(f \otimes g)(n))$.
\end{lem}

\begin{proof}
It suffices to check that $\langle \pi(\omega'_f \otimes \omega'_g), \Omega \rangle = 0$. We have
\begin{align*}
\langle \pi(\omega'_f \otimes \omega'_g), \Omega \rangle & = \langle G(\chi_f)^{-1}G(\chi_g)^{-1}(\alpha_f^+ \alpha_g^{(-1)^{n+1}} e_f^+ \otimes e_g^{(-1)^{n+1}} + \alpha_f^- \alpha_g^{(-1)^n} e_f^- \otimes e_g^{(-1)^n}), \\
& \qquad G(\overline{\chi_f}) G(\overline{\chi_g}) (\alpha_f^+ e_f^{-,\vee}+\alpha_f^- e_f^{+,\vee}) \otimes (-\alpha_g^+ e_g^{-,\vee}+\alpha_g^- e_g^{+,\vee}) \cdot (2\pi i)^{-k-\ell-2} \rangle\\
& = \left(\alpha_f^+ \alpha_g^{(-1)^{n+1}} \alpha_f^- (-1)^{n+1} \alpha_g^{(-1)^n} + \alpha_f^- \alpha_g^{(-1)^n} \alpha_f^+ (-1)^n \alpha_g^{(-1)^{n+1}}\right) \cdot \frac{G(\overline{\chi_f})G(\overline{\chi_g})}{G(\chi_f) G(\chi_g)} (2\pi i)^{-k-\ell-2} \\
& = 0.
\end{align*}
\end{proof}

\begin{lem}\label{Poincare}
We have $\langle t, \Omega \rangle = (-1)^{n+1} \chi_f(-1)\chi_g(-1) N_{\chi_f} N_{\chi_g} (2\pi i)^{k+\ell-2j}$.
\end{lem}
\begin{proof}
We have
\begin{align*}
\langle t, \Omega \rangle & = \langle G(\chi_f)G(\chi_g)(2\pi i)^{2n-2} (\alpha_f^- \alpha_g^{(-1)^n})^{-1} e_f^+ \otimes e_g^{(-1)^{n+1}},\\
&  \qquad G(\overline{\chi_f}) G(\overline{\chi_g}) (\alpha_f^+ e_f^{-,\vee}+\alpha_f^- e_f^{+,\vee}) \otimes (-\alpha_g^+ e_g^{-,\vee}+\alpha_g^- e_g^{+,\vee}) \cdot (2\pi i)^{k-\ell-2} \rangle\\
& = G(\chi_f)G(\overline{\chi_f})G(\chi_g)G(\overline{\chi_g})(2\pi i)^{k+\ell-2j} (\alpha_f^- \alpha_g^{(-1)^n})^{-1} \alpha_f^- (-1)^{n+1} \alpha_g^{(-1)^n} \\
& = (-1)^{n+1} \chi_f(-1)\chi_g(-1) N_{\chi_f} N_{\chi_g} (2\pi i)^{k+\ell-2j},
\end{align*}
since for any Dirichlet character $\chi$, we have $G(\chi) G(\overline{\chi}) = \chi(-1) N_\chi$.
\end{proof}

\section{Eisenstein symbols}\label{sec eisenstein}
Here we recall Beilinson's theory of the Eisenstein symbol \cite{Beilinson1986}.
Let $N \geq 3$ be an integer.
%Let $Y=Y(N)$ be the modular curve of full level $N$-structure defined over $\Q$ and
%$\pi :E\to Y $ the universal elliptic curve over $Y$.
%For a non-negative integer $k$, denote the $k$-fold fiber product of $E$ over $Y$ by $E^k$.
The complex points of $E^k$ are given by \cite[(3.4), (3.6)]{Deninger2}
$$
E^k(\C) \cong (\Z^{2k} \rtimes \SL_2(\Z)) \backslash \left(\h \times \C^k \times \GL_2(\Z/N\Z)\right).
$$
%$$
%E^k(\C)=\coprod_{(\Z\slash N\Z)^{\times}}\Gamma (N)\backslash \mathcal{H}\times \C^k \slash \Z^{2k},
%$$
where the action of $\SL_2(\Z)$ is given by
\[
\begin{pmatrix}
a & b\\
c & d
\end{pmatrix}
\cdot (\tau; z_1,\ldots ,z_k ; h)=\left( \frac{a\tau+b}{c\tau+d}; \frac{z_1}{c\tau+d}, \ldots , \frac{z_k}{c\tau+d} ; \begin{pmatrix} a & b \\ c & d \end{pmatrix} h\right)
\]
and the action of $\Z^{2k}$ is given by
\[
(u_1,v_1, \ldots , u_k,v_k)\cdot (\tau ; z_1, \ldots , z_k ; h)
=(\tau; z_1+u_1 - v_1 \tau ,\ldots, z_k + u_k - v_k \tau ; h).
\]
Let $\varepsilon_k$ be the signature character of $\mathfrak{S}_{k+1}$ on $E^k \subset E^{k+1}$.
For $i=0,\ldots ,k$, let $q_i$ denote the composition of
$E^k \hookrightarrow E^{k+1}\overset{\mathrm{pr}_i}{\to}E$.
Denote $\mathcal{U}_N=E\setminus E[N]$, where $E[N]$ is the $N$-torsion subgroup.
Write $\mathcal{U}_N^{(i)}=q_i^{-1}(\mathcal{U}_N)$ and
$\mathcal{U}_N'=\bigcap_{i=0}^{k}\mathcal{U}_N^{(i)}\subset E^k$.

Choose $g_0,\ldots,g_k \in \mathcal{O}(\mathcal{U}_N)^{\times}$.
Denote $g=q_0^*(g_0)\cup \cdots \cup q_k^*(g_k) \in H_\mathcal{M}^{k+1}(\mathcal{U}_N',\Q (k+1))$.
Write $\widetilde{G}=(\Z\slash N\Z )^{2k}\rtimes \mathfrak{S}_{k+1}$.
Here $(\Z\slash N\Z)^{2k}\simeq E[N]^k$ acts on $E[N]$ by the natural translation.
Let $\varepsilon_k : \widetilde{G} \to \{ \pm 1\}$ be the signature character defined by
$\varepsilon (g)=\varepsilon (\sigma )=\mathrm{sign}(\sigma)$
for $g=(t,\sigma)\in \widetilde{G}=(\Z\slash N\Z )^{2k}\rtimes \mathfrak{S}_{k+1}$.
Then $\widetilde{G}$ acts on $E^k$ and $\mathcal{U}_N'$.
This induces the action of $\widetilde{G}$ on the motivic cohomology $H_{\mathcal{M}}^{k+1}(\mathcal{U}_N',\Q (k+1))$.
Denote the idempotent corresponding to $\varepsilon_k$ by $\widetilde{e}_k$.
Hence we have the $\widetilde{e}_k$-eigenspace $H_{\mathcal{M}}^{k+1}(\mathcal{U}_N',\Q (k+1))^{\widetilde{e}_k}$
and the projection
$$
\mathrm{pr}_{\widetilde{e}_k}:H_{\mathcal{M}}^{k+1}(\mathcal{U}_N',\Q (k+1))\to
H_{\mathcal{M}}^{k+1}(\mathcal{U}_N',\Q (k+1))^{\widetilde{e}_k}
$$
defined by $x\mapsto |\widetilde{G}|^{-1}\sum_{g\in \widetilde{G}}\varepsilon_k(g)g\cdot x$.

Let $M$ be a positive auxiliary integer.
Let $j:\mathcal{U}_{MN}'\hookrightarrow \mathcal{U}_N'$ be the canonical inclusion
and $[\times M]:\mathcal{U}_{MN}'\to \mathcal{U}_N'$ the multiplication by $M$.
Then $j$ and $[\times M]$ induce
$$
j^*:H_{\mathcal{M}}^{k+1}(\mathcal{U}_N',\Q (k+1))
\to H_{\mathcal{M}}^{k+1}(\mathcal{U}_{MN}',\Q (k+1))^{(\Z\slash M\Z)^{2k}}
$$
and
$$
[\times M]^*:H_{\mathcal{M}}^{k+1}(\mathcal{U}_{N}',\Q (k+1))
\overset{\sim}{\to} H_{\mathcal{M}}^{k+1}(\mathcal{U}_{MN}',\Q (k+1))^{(\Z\slash M\Z)^{2k}}.
$$
Write $[\times M^{-1}]=([\times M]^*)^{-1}\circ j^*$.
Denote by $H_{\mathcal{M}}^{k+1}(\mathcal{U}_N',\Q (k+1))^{\widetilde{e}_k}_k$ the maximal quotient of
$H_{\mathcal{M}}^{k+1}(\mathcal{U}_N',\Q (k+1))^{\widetilde{e}_k}$
such that $[\times M^{-1}]=M^{-k}$ for any $M\geq 1$.
Then we have a canonical projection
$$
\overline{\mathrm{pr}}_{\widetilde{e}_k} : 
H_{\mathcal{M}}^{k+1}(\mathcal{U}_N',\Q (k+1))\to
H_{\mathcal{M}}^{k+1}(\mathcal{U}_N',\Q (k+1))^{\widetilde{e}_k}_k .
$$
\begin{thm}[{\cite[(8.16) Theorem]{Deninger}}]
The canonical map
$$
\alpha^*:
H_{\mathcal{M}}^{k+1}(E^k,\Q (k+1))^{e_k}
=H_{\mathcal{M}}^{k+1}(E^k,\Q (k+1))^{\widetilde{e}_k}
\to H_{\mathcal{M}}^{k+1}(\mathcal{U}_N',\Q (k+1))^{\widetilde{e}_k}_k
$$
induced by $\alpha : \mathcal{U}_N'\hookrightarrow E^k$
is bijective.
\end{thm}
%\begin{proof}
%This is .
%\end{proof}
We write $\widetilde{\mathrm{Eis}}^k(g_0,\ldots,g_k)=(\alpha^*)^{-1}(\overline{\mathrm{pr}}_{e_k}(g))
\in H^{k+1}_{\mathcal{M}}(E^k,\Q(k+1))$ for $g_0,\ldots,g_k \in \mathcal{O}(\mathcal{U}_N)^{\times}$.
In fact, $\widetilde{\mathrm{Eis}}^k$ factors through the divisors $\Q[(\Z\slash N\Z)^2]^0$.
Therefore we have a commutative diagram:
\[
\begin{CD}
\bigotimes_{i=0}^{k}\mathcal{O}(\mathcal{U}_N)^{\times}@>\widetilde{\mathrm{Eis}}^k>> H^{k+1}_{\mathcal{M}}(E^k,\Q(k+1))\\
@V\mathrm{Div}VV                                          @AA\mathrm{Eis}^k A \\
\bigotimes_{i=0}^k\Q[(\Z\slash N\Z )^2]_{\varepsilon_k}^0 @<\theta< \simeq < \Q [(\Z\slash N\Z)^2]^0
\end{CD}
\]
where $\theta$ is defined by $\beta \mapsto [\beta \otimes \alpha \otimes \cdots \otimes \alpha]$
with $\alpha =N^2[0]-\sum_{x\in (Z\slash N\Z)^{2}}[x]$.
The map
$$\mathrm{Eis}^k:\Q [(\Z\slash N\Z)^2]^0 \to H^{k+1}_{\mathcal{M}}(E^k,\Q(k+1))^{e_k}$$
is called the Eisenstein symbol.
For a smooth projective variety $X$ over $\R$, let $H^i_{\mathcal{D}}(X,\R(j))$ denote its Deligne cohomology.

We now recall an explicit formula for the realization of the Eisenstein symbol. Fix an integer $k \geq 0$. Let
$$
r_{\mathcal{D}}:H^{k+1}_{\mathcal{M}}(E^k,\Q(k+1))^{e_k}\to
H^{k+1}_{\mathcal{D}}(E^k_{\R},\R(k+1))^{e_k}
$$
be the regulator map.

By \cite[(7.3)]{Nekovar}, the Deligne cohomology group is given by:
$$
H^{k+1}_{\mathcal{D}}(E^k_{\R},\R(k+1))\simeq
\frac{\{\varphi \in H^0(E^k_{\R,\mathrm{an}}, \mathcal{A}^k\otimes \R(k)) \,| \, d\varphi=\frac12 (\omega+(-1)^k \overline{\omega}), \omega \in \Omega^{k+1}(\overline{\overline{E}}^k) \langle D \rangle)\} }{dH^0(E^k_{\R,\mathrm{an}},\mathcal{A}^{k-1}\otimes \R(k))},
$$
where $\mathcal{A}^{\cdot}$ is the de Rham complex of real valued $C^{\infty}$-forms, $\overline{\overline{E}}^k$ is a smooth compactification of $E^k(\C)$ and $D=\overline{\overline{E}}^k \setminus E^k(\C)$.

Recall that
$$
E^k(\C) \cong (\Z^{2k} \rtimes \SL_2(\Z)) \backslash \left(\h \times \C^k \times \GL_2(\Z/N\Z)\right).
$$
Write $\tau$ (resp. $z_1,\ldots,z_k$) for the coordinate on $\h$ (resp. $\C^k$). Write $h$ for an element of $\GL_2(\Z/N\Z)$. For any integer $0 \leq j \leq k$, define
$$
\psi_{k,j} = \frac{1}{k!} \sum_{\sigma \in \mathfrak{S}_k} \varepsilon(\sigma) d\overline{z}_{\sigma(1)} \wedge \cdots \wedge d\overline{z}_{\sigma(j)} \wedge dz_{\sigma(j+1)} \wedge \cdots \wedge dz_{\sigma(k)}.
$$
Let $\beta \in \Q[(\Z\slash N\Z)^2]^0$. Then by \cite[(3.12), (3.28)]{Deninger2} and \cite[Remark after Lemma 7.1]{Huber-Kings}, $r_{\mathcal{D}}(\mathrm{Eis}^k(\beta))$ is represented by
$$
\Phi^k(\beta) :=
-\frac{k!(k+2)}{N(2\pi i)}\cdot \frac{\tau-\overline{\tau}}{2} \sum_{a=0}^k \psi_{k,a} \cdot \left(\sideset{}{'}\sum_{(c,d) \in \Z^2} \sum_{v \in (\Z/N\Z)^2} \frac{\beta(h^{-1} v) \cdot e^{\frac{2\pi i(cv_1+dv_2)}{N}}}{(c \tau +d)^{k+1-a}(c \overline{\tau}+d)^{a+1}}\right) \mod{d\tau,d\overline{\tau}}
$$
where $\sum'$ denotes that we omit the term $(c,d)=(0,0)$. For brevity, for any $a,b \geq 1$ we put
$$
\mathcal{E}^{a,b}_{\beta}(\tau,h) := \sideset{}{'}\sum_{(c,d) \in \Z^2} \sum_{v \in (\Z/N\Z)^2} \frac{\beta(h^{-1} v) \cdot e^{\frac{2\pi i(cv_1+dv_2)}{N}}}{(c \tau +d)^a (c \overline{\tau}+d)^b}.
$$

\section{Construction of elements in the motivic cohomology}\label{construction}

Let $k,\ell$ be non-negative integers with $k\leq \ell$ and choose an integer $j$ such that $0\leq j \leq k$.
Write $k'=k-j \geq 0$ and $\ell' = \ell -j \geq 0$.
Consider the following three morphisms:
\begin{enumerate}
\item $p:E^{k'+j+\ell'}\to E^{k'+\ell'}$ given by 
$$(\tau;u_1,\ldots,u_{k'},t_1,\ldots,t_j,v_1,\ldots,v_{\ell'}; h)
\mapsto (\tau;u_1,\ldots,u_{k'},v_1,\ldots,v_{\ell'}; h).
$$
\item $\Delta :E^{k'+j+\ell'}\to E^{k'+2j+\ell'}=E^{k+\ell}$ given by 
$$
(\tau;u_1,\ldots,u_{k'},t_1,\ldots,t_j,v_1,\ldots,v_{\ell'}; h)
\mapsto (\tau;u_1,\ldots,u_{k'},t_1,\ldots,t_j,t_1,\ldots, t_j,v_1,\ldots,v_{\ell'}; h).
$$
\item $i:E^{k'+2j+\ell'}=E^{k+\ell} \to E^k\times E^\ell$ given by 
$$
(\tau;u_1,\ldots,u_{k'},t_1,\ldots,t_j,t_1',\ldots, t_j',v_1,\ldots,v_{\ell'}; h) \mapsto
((\tau;u_1,\ldots,u_{k'},t_1,\ldots,t_j ; h),(\tau;t_1',\ldots,t_j',v_1,\ldots,v_{\ell'} ; h)).
$$
\end{enumerate}
Note that $(i\circ \Delta)(\tau; u,t,v; h)=((\tau;u,t; h),(\tau ;t,v; h))$.

\begin{defn}
For $\beta \in \Q[(\Z\slash N\Z)^2]^0$, denote by $\Xi^{k,\ell,j}(\beta)$
the image of $\mathrm{Eis}^{k'+\ell'}(\beta)$ under the composite of morphisms:
\begin{align*}
H^{k'+\ell'+1}_{\mathcal{M}}(E^{k'+\ell'},\Q (k'+\ell'+1))%^{e_{k'+\ell'}}
&\overset{p^*}{\longrightarrow}
H^{k'+\ell'+1}_{\mathcal{M}}(E^{k'+j+\ell'},\Q (k'+\ell'+1))\\
&\overset{\Delta_*}{\longrightarrow}
H^{k+\ell+1}_{\mathcal{M}}(E^{k+\ell},\Q (k+\ell-j+1))\\
&\overset{i_*}{\longrightarrow}
H^{k+\ell+3}_{\mathcal{M}}(E^{k} \times E^\ell,\Q (k+\ell-j+2)).
%&\overset{\mathrm{pr}_{(e_k,e_{\ell})}}{\longrightarrow}
%H^{k+\ell+3}_{\mathcal{M}}(E^{k} \times E^\ell,\Q (k+\ell-j+2))^{(e_k,e_{\ell})},
\end{align*}
\end{defn}
%where $\mathrm{pr}_{(e_{k},e_{\ell})}$ is the projection to $(e_k,e_{\ell})$-eigenspace.
%Also write $\Xi^{k,\ell,j}_{\mathcal{D}}=r_{\mathcal{D}}(\Xi^{k,\ell,j}(\beta))$, where
%$$
%r_{\mathcal{D}}:H^{k+\ell+3}_{\mathcal{M}}(E^{k} \times E^\ell,\Q (k+\ell-j+2))
%\to H^{k+\ell+3}_{\mathcal{D}}(E^{k} \times E^\ell ,\R (k+\ell-j+2))
%$$
%is the regulator map.
%
%Let $f\in S_{k+2}(\Gamma_1(N_f), \chi_f)^{\mathrm{new}}$, $g\in S_{\ell+2}(\Gamma_1(N_g),\chi_g)^{\mathrm{new}}$
%be normalized eigenforms with Fourier expansion
%$$
%f(\tau)=\sum_{n=1}^{\infty}a_n(f)e^{2\pi i n \tau}, \,\, g(\tau)=\sum_{n=1}^{\infty}a_n(g)e^{2\pi i n \tau}.
%$$
%Write $g^*(\tau)=\sum_{n=1}^{\infty}\overline{a_n(g)}e^{2\pi i n\tau}$.
%Then we have $g^* \in S_{\ell+2}(\Gamma_1(N_g),\overline{\chi_g})$.
%We denote
%$$
%\omega_f=(2\pi i)^{k+1} f(\tau) d\tau\wedge dz_1\wedge \cdots \wedge dz_k, \,
%\omega_{g^*}=(2\pi i)^{\ell+1} g^*(\tau)d\tau \wedge dz_1 \wedge \cdots \wedge dz_{\ell}.
%$$
%By Hodge theory, we have
%$$0\to \mathrm{Fil}^{k+\ell-j+2} H^{k'+\ell'}_{\mathrm{dR}}((E^{k} \times E^\ell)_{\slash \R})
%\to H^{k'+\ell'}_{\mathrm{B}}((E^{k} \times E^\ell)_{\slash \R},\R (k+\ell-j+1))
%\to H^{k'+\ell'+1}_{\mathcal{D}}((E^{k} \times E^\ell)_{\slash \R},\R (k+\ell-j+2))
%\to 0
%$$

\section{The Rankin-Selberg method}\label{Rankin-Selberg}

Let $L(f \otimes g,s)$ denote the $L$-function associated to the $4$-dimensional Galois representation $V_f \otimes V_g$. We have
$$
L(f \otimes g,s) = \prod_{p \textrm{ prime}} P_p(f \otimes g,s)^{-1},
$$
where $P_p(f \otimes g,s)=\mathrm{det}(1-\mathrm{Frob}_p \cdot p^{-s}|(V_f\otimes V_g)^{I_p})$
is a polynomial in $p^{-s}$.
Then the polynomial $P_p(f\otimes g,s)$ coincides up to the shift $s \mapsto s-\frac{k+\ell+2}{2}$ with the automorphic $L$-factor defined by Jacquet in \cite{Jacquet}, and $L(f \otimes g,s)$ converges for $\mathrm{Re}(s)>\frac{k+\ell}{2}+2$.
%Consider the completed $L$-function
%$$
%\Lambda(f \otimes g,s) = \Gamma_{\C}(s) \Gamma_{\C}(s-k+1) L(f \otimes g,s)
%$$
%where $\Gamma_\C(s) = (2\pi)^{-s} \Gamma(s)$.

Let $N$ be an integer divisible by $N_f$ and $N_g$. Let $\chi : (\Z/N\Z)^\times \to \C^\times$ be the Dirichlet character induced by $\chi_f \chi_g$. Put $D(f,g,s):= \sum_{n=1}^{\infty} a_n(f) a_n(g) n^{-s}$. By \cite[Lemma 1]{Shimura}, we have
$$
L(\chi,2s-k-\ell-2) D(f,g,s) = R_{f,g,N}(s) L(f \otimes g,s),
$$
where
$$
R_{f,g,N}(s) := \left(\prod_{p | N} P_p(f \otimes g, s) \right) \sum_{n \in S(N)} \frac{a_n(f) a_n(g)}{n^s}
$$
is a polynomial in the variables $p^{-s}$ for $p|N$ by \cite[Theorem 15.1]{Jacquet}. Here $S(N)$ denotes the set of integers all of whose prime factors divide $N$.

For any Dirichlet character $\omega : (\Z/N\Z)^\times \to \C^\times$, define the Eisenstein series
$$
E_{\ell-k,N}(\tau,s,\omega) = \sideset{}{'}\sum_{m,n \in \Z} \frac{\omega(n)}{(Nm\tau+n)^{\ell-k} |Nm\tau+n|^{2s}}.
$$

% Consider the following Rankin product of $f$ and $g$ with a real-analytic Eisenstein series
%$$
%R(f,g,\alpha,s) = \int_{\Gamma_1(N) \backslash \mathcal{H}} f(\tau) g(-\overline{\tau}) E_\alpha^{(\ell-k)}(\tau,s-\ell-1) y^\ell \mathrm{d}x \mathrm{d}y.
%$$

\begin{thm}[Shimura {\cite[(2.4)]{Shimura}}]\label{thm shimura}
We have
$$
\int_{\Gamma_0(N) \backslash \mathcal{H}} f(\tau) g(-\overline{\tau}) E_{\ell-k,N}(\tau,s-1-\ell,\chi) y^{s-1} dx dy = 2 (4\pi)^{-s} \Gamma(s) L(\chi,2s-k-\ell-2) D(f,g,s).
$$
\end{thm}
\begin{rem}
Let us assume $k=\ell$. Then by \cite[(2.5)]{Shimura} and \cite[page 220, Correction]{Shimura2}, $D(f,g,s)$ has a pole at $s=k+2$ if and only if $\langle {f}^*,g\rangle \neq 0$.
This is equivalent to $g={f}^*$. In this case, we have $\chi_g =\chi_f^{-1}$, hence $\chi$ is trivial.
Therefore, our assumption $\chi \neq 1$ excludes the case where $L(f\otimes g,s)$ has a pole.
\end{rem}
%\begin{thm}[Rankin-Selberg, Shimura]\label{thm shimura}
%We have
%$$
%R(f,g,1/N,s) = 2^{-1-\ell} i^{\ell-k} N^{2s-2-k-\ell} C(s) \Lambda(f \otimes g,s)
%$$
%where
%$$
%C(s) = \Bigl(\prod_{p | N} P_p(f \otimes g, s) \Bigr) \sum_{n \in S(N)} \frac{a_n(f) a_n(g)}{n^s}.
%$$
%Here $S(N)$ denotes the set of integers all of whose prime factors divide $N$.
%\end{thm}

\section{Computation of the regulator integral}\label{computation}

Let $j$ be an integer satisfying $0 \leq j \leq k \leq \ell$. Recall that we have a differential form $\Omega_{f,g} := \omega_{f^*} \otimes \overline{\omega_g} \in H^{k+\ell+2}_{\dR}(M(f^*\otimes g^*)(j+1))\otimes \mathbb{C}$. Since $M(f^* \otimes g^*)$ is a direct factor of $h^{k+\ell+2}(E^k \times E^\ell) \otimes K_{f,g}$, we may consider $\Omega_{f,g}$ as an element of $H^{k+\ell+2}_{\dR}(E^k \times E^\ell) \otimes K_{f,g} \otimes \C$. By the same argument as in Lemma \ref{pairing Omega}, and since $\Omega_{f,g}$ has rapid decay at infinity, pairing with $\Omega_{f,g}$ yields a linear map
$$
\langle \, \cdot \, ,\Omega_{f,g} \rangle : H^{k+\ell+3}_{\mathcal{D}}(E^k_{\R} \times E^{\ell}_{\R}, \mathbb{R}(n)) \to K_{f,g} \otimes \C.
$$

%We have a duality pairing
%$$
%H^{k+\ell+2}(E^k(\C) \times E^{\ell}(\C),\C) \times H^{k+\ell+2}_c(E^k(\C) \times E^{\ell}(\C),\C) \to \C.
%$$
%Let $p_1 : E^k \times E^\ell \to E^k$, $p_2 : E^k \times E^\ell \to E^\ell$ be the canonical projections, and consider the differential form $\Omega_{f,g^*} = p_1^* \omega_f \wedge p_2^* \overline{\omega_{g^*}} \in H^{k+\ell+2}_c(E^k(\C) \times E^{\ell}(\C),\C)$. Integrating against $\Omega_{f,g^*}$ yields a linear map
%$$
%\langle \cdot , \Omega_{f,g^*} \rangle : H^{k+\ell+2}(E^k(\C) \times E^{\ell}(\C),\C) \to \C.
%$$
Let $\beta \in \Q[(\Z/N\Z)^2]^0$. In this section, we compute $\langle r_{\mathcal{D}}(\Xi^{k,\ell,j}(\beta)), \Omega_{f,g} \rangle$ in terms of the Rankin-Selberg $L$-function of $f$ and $g$. At the beginning $\beta$ is arbitrary, but from Definition \ref{def betachi} on, we will use a particular choice of $\beta$.

\begin{lem}\label{lem rDXi}
We have
$$
\langle r_{\mathcal{D}}(\Xi^{k,\ell,j}(\beta)), \Omega_{f,g} \rangle = \frac{1}{(2\pi i)^{k'+j+\ell'+1}} \int_{E^{k'+j+\ell'}(\C)} p^* \Phi^{k'+\ell'}(\beta) \wedge \Delta^* i^* \Omega_{f,g}.
$$
\end{lem}

\begin{proof}
We have
\begin{align*}
\langle r_{\mathcal{D}}(\Xi^{k,\ell,j}(\beta)), \Omega_{f,g} \rangle & = \langle r_{\mathcal{D}}(i_* \Delta_* p^* \mathrm{Eis}^{k'+\ell'}(\beta)), \Omega_{f,g} \rangle\\
& = \langle r_{\mathcal{D}}(p^* \mathrm{Eis}^{k'+\ell'}(\beta)), \Delta^* i^* \Omega_{f,g} \rangle\\
%& = G(\chi_f)G(\chi_f)\langle r_{\mathcal{D}}(p^* \mathrm{Eis}^{k'+\ell'}(\beta)), \Delta^* i^* (\omega_{f^*} \otimes \overline{\omega_g}) \rangle\\
& = \frac{1}{(2\pi i)^{\mathrm{dim}E^{k'+j+\ell'}}} \int_{E^{k'+j+\ell'}(\C)} p^* r_{\mathcal{D}}( \mathrm{Eis}^{k'+\ell'}(\beta)) \wedge \Delta^* i^* \Omega_{f,g}\\
& = \frac{1}{(2\pi i)^{k'+j+\ell'+1}} \int_{E^{k'+j+\ell'}(\C)} p^* \Phi^{k'+\ell'}(\beta) \wedge \Delta^* i^* \Omega_{f,g}.
\end{align*}
\end{proof}

Let $\tau, z_1,\ldots,z_{k+\ell-j}$ denote the coordinates on $E^{k+\ell-j}(\C)$.
%
%\begin{lem}
%We have
%\begin{align*}
%& p^* \Phi^{k'+\ell'}(\beta) \wedge \Delta^* i^* (\omega_{f^*} \otimes \overline{\omega_g}) \\
%& = -\frac{C(k'+\ell'+2)!}{2\pi iN(k'+\ell'+1)}
%f^*(\tau) \overline{g(\tau)} {\sum_{c_1,c_2 \in \Z}}'
%\frac{\psi_{\beta}(c_1,c_2) \mathrm{Im}(\tau)}{(c_1\tau +c_2)^{\ell'+1}(c_1\overline{\tau}+c_2)^{k'+1}} d\tau \wedge d\overline{\tau} \wedge \bigwedge_{i=1}^{k+\ell-j} dz_i \wedge d\overline{z}_i
%\end{align*}
%\end{lem}
Note that the differential form
$$
\Delta^* i^* \Omega_{f,g} = (-1)^{k+\ell+1} (2\pi i)^{k+\ell+2} f^*(\tau) \overline{g(\tau)} d\tau \wedge d\overline{\tau} \wedge dz_1 \wedge \cdots \wedge dz_k  \wedge d\overline{z}_{k-j+1} \wedge \cdots \wedge d\overline{z}_{k+\ell-j}
$$
already contains $d\tau \wedge d\overline{\tau}$. Therefore, we may neglect the terms of $\Phi^{k'+\ell'}(\beta)$ involving $d\tau,d\overline{\tau}$. Moreover, we have
$$
p^* \psi_{k'+\ell',a} \wedge \Delta^* i^* \Omega_{f,g} = \begin{cases} C_1 f^*(\tau) \overline{g(\tau)} d\tau \wedge d\overline{\tau} \wedge \bigwedge_{i=1}^{k+\ell-j} dz_i \wedge d\overline{z}_i & \textrm{if } a=k',\\
0 & \textrm{if } a \neq k',
\end{cases}
$$
with
$$
C_1 = (-1)^{k+\ell+1+{k'}^2+j(k'+\ell')+(k'+j+\ell')(k'+j+\ell'-1)/2} \frac{{k'}! \cdot {\ell'}!}{(k'+\ell')!} (2\pi i)^{k+\ell+2}.
$$
%$$
%C_1 = \pm \frac{{k'}! \cdot {\ell'}!}{(k'+\ell')!} (2\pi i)^{k+\ell+2}.
%$$
It follows that
\begin{align*}
& p^* \Phi^{k'+\ell'}(\beta) \wedge \Delta^* i^* \Omega_{f,g} \\
& = -\frac{(k'+\ell')!\cdot (k'+\ell'+2)}{N(2\pi i)} \cdot \frac{\tau-\overline{\tau}}{2} \cdot \mathcal{E}^{\ell'+1,k'+1}_{\beta}(\tau,h) \cdot p^*\psi_{k'+\ell',k'} \wedge \Delta^* i^* \Omega_{f,g}\\
& = -\frac{C_1 \cdot (k'+\ell')!\cdot ( k'+\ell'+2)}{N(2\pi i)} \cdot \frac{\tau-\overline{\tau}}{2} \cdot \mathcal{E}^{\ell'+1,k'+1}_{\beta}(\tau,h) \cdot f^*(\tau) \overline{g(\tau)} d\tau \wedge d\overline{\tau} \wedge \bigwedge_{i=1}^{k+\ell-j} dz_i \wedge d\overline{z}_i.
\end{align*}
Recall \cite[(3.4)]{Deninger2} that the complex points of $Y(N)$ are given by
$$
Y(N)(\C) = \SL_2(\Z) \backslash (\h \times \GL_2(\Z/N\Z)).
$$
Note that $\int_{\C / (\Z+\tau \Z)} dz \wedge d\overline{z} = -2i \mathrm{Im}(\tau)$. Using Lemma \ref{lem rDXi} and integrating over the fibers of $E^{k'+j+\ell'}$ over $Y(N)$, we get
$$
\langle r_{\mathcal{D}}(\Xi^{k,\ell,j}(\beta)), \Omega_{f,g} \rangle = -
\frac{(-2i)^{k+\ell-j}\cdot i \cdot C_1\cdot (k'+\ell'+2)!}{(2 \pi i)^{k+\ell-j+2}\cdot N\cdot (k'+\ell'+1)} \int_{Y(N)(\C)} f^*(\tau) \overline{g(\tau)} \mathcal{E}^{\ell'+1,k'+1}_{\beta}(\tau,h) \mathrm{Im}(\tau)^{k+\ell-j+1} d\tau \wedge d\overline{\tau}.
$$
We have an isomorphism of analytic spaces
\begin{align*}
\nu : (\Z/N\Z)^\times \times \Gamma(N) \backslash \h & \xrightarrow{\cong} Y(N)(\C)\\
(a,[\tau]) & \mapsto \left[\left(\tau, \begin{pmatrix} 0 & -1 \\ a & 0\end{pmatrix}\right)\right].
\end{align*}
Note that $\nu(a,\tau)$ corresponds to the elliptic curve $\C/(\Z+\tau \Z)$ with basis of $N$-torsion $(a\tau/N, 1/N)$ in the moduli space.

Let $\chi : (\Z/N\Z)^\times \to \C^\times$ be the Dirichlet character induced by $\chi_f \chi_g$. \emph{Assume $\chi \neq 1$}.

\begin{defn}\label{def betachi}
Let $\beta_\chi \in \Q(\chi)[(\Z/N\Z)^2]^0 \subset K_{f,g}[(\Z/N\Z)^2]^0$ be the divisor defined by
$$
\beta_\chi(v_1,v_2) = \begin{cases} \overline{\chi}(-v_2) & \textrm{if } v_1 = 0,\\
0 & \textrm{if } v_1 \neq 0.
\end{cases}
$$
\end{defn}

For an integer $w \geq 0$, $\alpha \in \Q/\Z$, $\tau \in \h$ and $s \in \C$, define the following standard real-analytic Eisenstein series as in \cite[Definition 4.2.1]{LLZ}:
$$
E^{(w)}_\alpha(\tau,s)=(-2\pi i)^{-w} \pi^{-s} \Gamma(s+w) \sideset{}{'}\sum_{m,n \in \Z} \frac{\mathrm{Im}(\tau)^s}{(m\tau+n+\alpha)^w |m\tau+n+\alpha|^{2s}},
$$
where $\sideset{}{'}\sum$ denotes that the term $(m,n)=(0,0)$ is omitted if $\alpha=0$, and
$$
F^{(w)}_\alpha(\tau,s)=(-2\pi i)^{-w} \pi^{-s} \Gamma(s+w) \sideset{}{'}\sum_{m,n \in \Z} \frac{e^{2\pi i \alpha m}\mathrm{Im}(\tau)^s}{(m\tau+n)^w |m\tau+n|^{2s}},
$$
where $\sideset{}{'}\sum$ denotes that the term $(m,n)=(0,0)$ is omitted. For fixed $w,\alpha,\tau$, these functions have meromorphic continuations to the whole $s$-plane, and are holomorphic everywhere if $w \neq 0$. Note that
$$
\sum_{\alpha \in (\Z/N\Z)^\times} \omega(\alpha) E^{(\ell-k)}_{\alpha/N}(\tau,s) = (-2\pi i)^{-\ell+k} \pi^{-s} \Gamma(s+\ell-k) \mathrm{Im}(\tau)^s N^{\ell-k+2s} E_{\ell-k,N}(\tau,s,\omega).
$$

\begin{lem}\label{lemma Ebetachi}
For any $a \in (\Z/N\Z)^\times$, we have
\begin{equation}\label{lem Ebetachi}
\mathcal{E}_{\beta_\chi}^{\ell'+1,k'+1}\left(\tau,\begin{pmatrix} 0 & -1 \\ a & 0 \end{pmatrix}\right) = \frac{\pi^{k'+\ell'+1} }{\ell'! N^{k'+\ell'} \cdot \mathrm{Im}(\tau)^{k'+\ell'+1}} \lim_{s \to -\ell'} \Gamma(s+\ell-k) E_{\ell-k,N}(\tau,s,\overline{\chi}).
\end{equation}
\end{lem}

\begin{proof}
We have
\begin{align*}
\mathcal{E}_{\beta_\chi}^{\ell'+1,k'+1}\left(\tau,\begin{pmatrix} 0 & -1 \\ a & 0 \end{pmatrix}\right) & = \sideset{}{'}\sum_{(c,d) \in \Z^2} \sum_{v_1,v_2 \in \Z/N\Z} \frac{\beta_\chi(a^{-1} v_2,-v_1) \cdot e^{\frac{2\pi i(cv_1+dv_2)}{N}}}{(c \tau +d)^{\ell'+1} (c \overline{\tau}+d)^{k'+1}}\\
& = \sideset{}{'}\sum_{(c,d) \in \Z^2} \sum_{v_1 \in (\Z/N\Z)^\times} \frac{\overline{\chi}(v_1) \cdot e^{\frac{2\pi i cv_1}{N}}}{(c \tau +d)^{\ell'+1} (c \overline{\tau}+d)^{k'+1}}\\
& = \frac{(-2\pi i)^{\ell-k} \pi^{k'+1}}{\ell'! \cdot \mathrm{Im}(\tau)^{k'+1}} \sum_{v_1 \in (\Z/N\Z)^\times} \overline{\chi}(v_1) F^{(\ell-k)}_{v_1/N}(\tau,k'+1)\\
& = \frac{(-2\pi i)^{\ell-k} \pi^{k'+1}}{\ell'! \cdot \mathrm{Im}(\tau)^{k'+1}} \sum_{v_1 \in (\Z/N\Z)^\times} \overline{\chi}(v_1) E^{(\ell-k)}_{v_1/N}(\tau,-\ell') \qquad \textrm{by \cite[4.2.2(iv)]{LLZ}}  \\
& = \frac{\pi^{k'+\ell'+1} }{\ell'! N^{k'+\ell'} \cdot \mathrm{Im}(\tau)^{k'+\ell'+1}} \lim_{s \to -\ell'} \Gamma(s+\ell-k) E_{\ell-k,N}(\tau,s,\overline{\chi}).
\end{align*}
\end{proof}

Note that the right hand side of (\ref{lem Ebetachi}) is independent of $a \in (\Z/N\Z)^\times$. Therefore, the contributions of the regulator integral over each connected component of $Y(N)(\C)$ are equal, and we have
$$
\langle r_{\mathcal{D}}(\Xi^{k,\ell,j}(\beta_\chi)), \Omega_{f,g} \rangle =
\frac{C_2 \cdot \phi(N)}{(2\pi i)^{k+\ell-j+1}} \int_{\Gamma(N) \backslash \h} f^*(\tau) g^*(-\overline{\tau})  \mathrm{Im}(\tau)^{j} \lim_{s \to -\ell'} \Gamma(s+\ell-k) E_{\ell-k,N}(\tau,s,\overline{\chi}) dx dy
$$
with
$$
C_2=\frac{(-2i)^{k+\ell-j}\cdot  i\cdot \pi^{k'+\ell'}\cdot (k'+\ell'+2)! }{N^{k'+\ell'+1} \cdot \ell' ! \cdot (k'+\ell'+1)} \cdot C_1.
$$
Since the integrand is invariant under the group $\Gamma_0(N)$, this can be rewritten as
$$
\langle r_{\mathcal{D}}(\Xi^{k,\ell,j}(\beta_\chi)), \Omega_{f,g} \rangle = \frac{C_2 \cdot N \cdot \phi(N)^2}{2 (2\pi i)^{k+\ell-j+1}} \lim_{s \to -\ell'} \Gamma(s+\ell-k) \int_{\Gamma_0(N) \backslash \h} f^*(\tau) g^*(-\overline{\tau})  E_{\ell-k,N}(\tau,s,\overline{\chi}) y^{s+\ell}  dx dy.
$$
Using Theorem \ref{thm shimura} with $f^*$ and $g^*$, we get
\begin{align*}
\langle r_{\mathcal{D}}(\Xi^{k,\ell,j}(\beta_\chi)), \Omega_{f,g} \rangle & = \frac{C_2 \cdot N\cdot \phi(N)^2}{2(2\pi i)^{k+\ell-j+1}} \lim_{s \to -\ell'} \Gamma(s+\ell-k)\cdot 2\cdot (4\pi)^{-s-1-\ell} \Gamma(s+1+\ell) \cdot \\
& \qquad \qquad \qquad \qquad \cdot {R}_{f^*,g^*,N}(s+1+\ell)L(f^* \otimes g^*,s+1+\ell)\\
& = \frac{C_2\cdot N\cdot \phi(N)^2 }{(2\pi i)^{k+\ell-j+1}}(4\pi)^{-j-1} \cdot j!\cdot {R}_{f^*,g^*,N}(j+1) \lim_{s \to -\ell'} \Gamma(s+\ell-k) L(f^* \otimes g^*,s+1+\ell)\\
& = \frac{C_2\cdot N\cdot \phi(N)^2}{(2\pi i)^{k+\ell-j+1}} (4\pi)^{-j-1} \cdot j!\cdot {R}_{f^*,g^*,N}(j+1) \frac{(-1)^{k-j}}{(k-j)!} L'(f^* \otimes g^*,j+1).
\end{align*}
Putting everything together, we have the following theorem.
\begin{thm}\label{regulator}
Let $\Omega_{f,g} = \omega_{f^*} \otimes \overline{\omega_g}$. Let $\chi : (\Z/N\Z)^\times \to \C^\times$ be the Dirichlet character induced by $\chi_f \chi_g$. Assume $\chi \neq 1$. Then we have the following identity in $K_{f,g} \otimes \C$
$$
\langle r_{\mathcal{D}}(\Xi^{k,\ell,j}(\beta_\chi)),\Omega_{f,g} \rangle
=\pm (2\pi i)^{k+\ell-2j}\cdot  \frac{(k+\ell -2j+2) \cdot j! \cdot \phi(N)^2}{2\cdot  N^{k+\ell-2j}}
\cdot {R}_{f^*,g^*,N}(j+1)\cdot L'(f^*\otimes g^*,j+1).
$$
\end{thm}
Note that $R_{f^*, g^*,N}(j+1)$ is an element of $K_{f,g}$.

\section{Computation of residues}\label{residues}
In this section, we extend the motivic element $\Xi^{k,\ell,j}(\beta)$ to the N\'eron model by computing the residue.
%\section{Motivic cohomology and Kuga-Sato varieties}
\subsection{Voevodsky's category of motives and motivic cohomology}
For a field $k$, let $DM_{gm}^{\mathrm{eff}}(k)$ be the category of effective geometrical motives over $k$.
For a scheme $X$ over $k$, we have the motive $M_{gm}(X)$ and the motive with compact support $M_{gm}^c(X)$.
We consider the $\Q$-linear analogue of $DM_{gm}^{\mathrm{eff}}(k)$ denoted by $DM_{gm}^{\mathrm{eff}}(k)_{\Q}$.
For any object $M$ of $DM_{gm}^{\mathrm{eff}}(k)_{\Q}$, we define the motivic cohomology by
$$
H_{\mathcal{M}}^i(M,\Q (j))=\mathrm{Hom}_{DM_{gm}^{\mathrm{eff}}(k)_{\Q}}(M,\Q(j)[i]).
$$
Then it is known that
$$
H_{\mathcal{M}}^i(M_{gm}(X),\Q (j)) \simeq H_{\mathcal{M}}^i(X,\Q (j)) \simeq CH^j(X,2j-i)
$$
for a smooth separated scheme $X$ over $k$, where $CH^n(X,m)$ is Bloch's higher Chow group.
\subsection{Motives for Kuga-Sato varieties}
Let $Y=Y(N)$ and $X=X(N)$.
Denote $X^{\infty}=X\setminus Y$.
The symmetric group $\mathfrak{S}_{k}$ acts on $\overline{E}^k$ by permutation,
$(\Z\slash N\Z )^{2k}$ by translations, and $\mu_2^k$ by inversion in the fiber.
Therefore we have the action of $G=((\Z\slash N\Z )^{2}\rtimes \mu_2)^k\rtimes \mathfrak{S}_k$.
This action can be extended to $\overline{\overline{E}}^k$.
Let $\varepsilon_k : G \to \{ \pm 1\}$ be the character which is trivial on $(\Z \slash N\Z )^{2k}$,
is the product on $\mu_2^k$, and is the sign character on $\mathfrak{S}_k$.
Then define the idempotent
$$
e_k:=\frac{1}{(2N^2)^k\cdot k!}\sum_{g\in G}\varepsilon_k(g)^{-1}\cdot g \in \Z [\frac{1}{2N\cdot k!}][G].
$$
Let $M_{gm}(\overline{\overline{E}}^k)^{e_k}\in DM_{gm}^{\mathrm{eff}}(\Q)_{\Q}$ be the image of
the idempotent $e_k$ on $M_{gm}(\overline{\overline{E}}^k)$.
Also denote by $M_{gm}(E^k)^{e_k}$ and $M_{gm}^c(E^k)^{e_k}$ the images of $e_k$ on $M_{gm}(E^k)$
and $M_{gm}^c(E^k)$ respectively.
Write the complement of $E^k$ in the smooth proper scheme $\overline{\overline{E}}^k$
by $\overline{\overline{E}}^{k,\infty}$.
%By \cite[Proposition 4.1.5]{Voevodsky}, we have the distinguished triangle
%$$
%M_{gm}(\overline{\overline{E}}^k)^{e_k}\to M_{gm}^c(E^k)^{e_k}\to
%M_{gm}(\overline{\overline{E}}^{k,\infty})^{e_k}[1] \to M_{gm}(\overline{\overline{E}}^k)^{e_k}[1].
%$$

Now we recall a result of Schappacher-Scholl \cite{Schappacher-Scholl}.
Fix an integer $N \geq 3$ and an integer $k \geq 0$.
Recall $X=X(N)$ is the compactified modular curve of level $N$ and $\overline{E}\to X$ the universal
generalized elliptic curve over $X$.
Consider the $k$-fold fiber product $\overline{E}^k=\overline{E}{\times}_X \cdots {\times}_X \overline{E}$
of $\overline{E}$ over $X$.
Denote $X^{\infty}=X \setminus Y$, where $Y=Y(N)$ is the modular curve of level $N$.
Let $\hat{E}^k$ be the N\'eron model of $E^k$ over $X$ and
$\overline{\overline{E}}^k \to \overline{E}^k$ Deligne's desingularization.
Then $\overline{\overline{E}}^k$ is a smooth projective variety over $\Q$.

By the generalized Gysin distinguished triangle
$$
M_{gm}({{E}}^k)^{e_k}\to M_{gm}({\overline{\overline{E}}}^k)^{e_k}\to
M_{gm}^c(\overline{\overline{E}}^{k,\infty})^{e_k}(1)[2] \to M_{gm}(E^k)^{e_k}[1],
$$
we get the localization sequence for the pair $(\overline{\overline{E}}^k, E^k)$:
$$
0\to H^{k+1}_{\mathcal{M}}(\overline{\overline{E}}^k,\mathbb{Q}(k+1))^{e_k}
\to H^{k+1}_{\mathcal{M}}({E}^k,\mathbb{Q}(k+1))^{e_k}
\overset{\mathrm{Res}^k}{\to}\mathcal{F}^k_N \to 0,
$$
where 
$\mathcal{F}^k_N \simeq H_{\mathcal{M}}^{k}(\overline{\overline{E}}^{k,\infty},\Q (k))^{e_k}\simeq H^0_{\mathcal{M}}(X^{\infty},\mathbb{Q}(0))
\simeq \mathbb{Q} [X^{\infty} ]$
is defined by
$$
\mathcal{F}^k_N=\left\{
f:\mathrm{GL}_2(\mathbb{Z}\slash N\mathbb{Z})\to \mathbb{Q} \,\middle| \, 
f(g\cdot
\begin{pmatrix}
a & b\\
0 &  1
\end{pmatrix})
=(- 1)^k f(-g) 
\textup{ for all } a \in (\mathbb{Z}\slash N\mathbb{Z})^{\times} \textup{ and } b\in \mathbb{Z} \slash N \mathbb{Z} 
\right\}.
$$
Then $\mathrm{Res}^k$ is $\mathrm{GL}_2(\Z\slash N\Z )$-equivariant.
%\vspace{-2mm}
Define the horospherical map $\omega^k_N : \mathbb{Q}[(\mathbb{Z}\slash N\mathbb{Z})^2]^0\to \mathcal{F}^k_N$ by
$$
\omega^k_N(\beta)(g)=\sum_{x=(x_1,x_2)\in (\mathbb{Z}\slash N\mathbb{Z})^2}
\beta(g\cdot^tx)B_{k+2}(\left\langle \frac{x_2}{N}\right\rangle )
$$
for $\beta\in \Q [(\Z\slash N\Z )^2]^0$ and $g\in\mathrm{GL}_2(\Z\slash N\Z)$, where $B_k$ is Bernoulli polynomial.
\begin{thm}[Schappacher-Scholl {\cite[7.2]{Schappacher-Scholl}}]
$
\mathrm{Res}^k\circ \mathrm{Eis}^k$ is a nonzero multiple of $\omega^k_N$.
%with $C_k=\frac{k+1}{(k+2)!}$.
\end{thm}
%\begin{rem}
%Note that our normalization of the horospherical map $\omega_N^k$ is same with Deninger \cite{Deninger2}
%which is slightly different from Schappacher-Scholl
%\cite{Schappacher-Scholl}.
%\end{rem}

We denote
$Z^k=\hat{E}^{k,*}\setminus E^k=\hat{E}^{k,*}\times_X X^{\infty} \simeq \mathbb{G}_m^k \times_\Q X^{\infty}$ (non-canonically),
$E^{k,\ell}=E^{k}\times E^{\ell}$, $\hat{E}^{k,\ell,*}=\hat{E}^{k,*}\times \hat{E}^{\ell,*}$,
$Z^{k,\ell}=Z^{k}\times Z^{\ell}$
and
$U^{k,\ell}=\hat{E}^{k,\ell,*} \setminus Z^{k,\ell}$. 
Let $i': E^{k+\ell}\to U^{k,\ell}$ be the canonical closed immersion.
Then $i'$ induces the morphism
$$
i'_*:H^{k+\ell+1}_{\mathcal{M}}(E^{k+\ell},\mathbb{Q}(k+\ell-j+1))
\to H^{k+\ell+3}_{\mathcal{M}}(U^{k,\ell},\mathbb{Q}(k+\ell-j+2)).
$$
Recall that we defined the morphisms:
$$
\begin{CD}
E^{k+\ell-j}@>\Delta >> E^{k+\ell} @>i >> E^k \times E^{\ell}\\
@VVpV \\
E^{k+\ell -2j}.
\end{CD}
$$
Similarly we define the morphisms
$$
\begin{CD}
\hat{E}^{k+\ell-j,*}@>\hat{\Delta} >> \hat{E}^{k+\ell,*} @>\hat{i} >> \hat{E}^{k,*} \times \hat{E}^{\ell,*}\\
@VV\hat{p}V \\
\hat{E}^{k+\ell -2j,*}
\end{CD}
$$
and
$$
\begin{CD}
Z^{k+\ell-j}@>\Delta_\infty >> Z^{k+\ell} @>i_\infty >> Z^k \times Z^{\ell}\\
@VVp_\infty V \\
Z^{k+\ell -2j}.
\end{CD}
$$
By \cite[Proposition 3.5.4]{Voevodsky}, for a smooth scheme $X$ and a smooth closed subscheme $Z$
of codimension $c$ we have the following Gysin distinguished triangle
$$
M_{gm}(X\setminus Z)\to M_{gm}(X) \to M_{gm}(Z)(c)[2c]\to M_{gm}(X\setminus Z)[1].
$$
Put $m=k+\ell -2j$. Then the diagram
$$
\small{
\begin{CD}
M_{gm}(E^{m})@>>> M_{gm}(\hat{E}^{m,*}) @>>> M_{gm}(Z^m)(1)[2]@>{+1}>> \\%M_{gm}(E^m)[1]\\
@AA{p_*}A @AA{\hat{p}_*}A @AA{p_{\infty,*}}A \\%@AAA \\
M_{gm}(E^{m+j})@>>> M_{gm}(\hat{E}^{m+j,*}) @>>> M_{gm}(Z^{m+j})(1)[2]@>{+1}>> \\%M_{gm}(E^{m+j})[1]\\
@AA{\Delta^*}A @AA{\hat{\Delta}^*}A @AA{\Delta_{\infty}^*}A \\%@AAA \\
M_{gm}(E^{m+2j})(-j)[-2j]@>>> M_{gm}(\hat{E}^{m+2j,*})(-j)[-2j] @>>> M_{gm}(Z^{m+2j})(-j+1)[-2j+2]@>{+1}>> \\%M_{gm}(E^{m+2j})(-j)[-2j+1]\\
@AA{{i'}^*}A @AA{\hat{i}^*}A @AA{{i_{\infty}}^*}A \\%@AAA \\
M_{gm}(U^{k,\ell})(-j-1)[-2j-2]@>>> M_{gm}(\hat{E}^{k,\ell,*})(-j-1)[-2j-2] @>>> M_{gm}(Z^{k,\ell})(-j+1)[-2j+2]@>{+1}>> \\%M_{gm}(E^{k,\ell})(-j-1)[-2j-1]
\end{CD}
}
$$
is commutative by \cite[Proposition 4.10, Theorem 4.32]{Deglise}. Taking cohomology, we get the following commutative diagram with exact rows:
$$
\begin{CD}
H_\mathcal{M}^{m+1}(\hat{E}^{m,*},\Q (m+1))@>>> H_\mathcal{M}^{m+1}(E^{m},\Q (m+1))@>>> H_\mathcal{M}^{m}(Z^{m},\Q (m))\\
@VV\hat{p}^* V @VVp^* V @VVp^*_\infty V \\
H_\mathcal{M}^{m+1}(\hat{E}^{m+j,*},\Q (m+1))@>>> H_\mathcal{M}^{m+1}(E^{m+j},\Q (m+1))@>>> H_\mathcal{M}^{m}(Z^{m+j},\Q (m))\\
@VV\hat{\Delta}_* V @VV\Delta_* V @VV\Delta_{\infty ,*}V \\
H_\mathcal{M}^{k+\ell +1}(\hat{E}^{k+\ell,*},\Q (m+j+1))@>>> H_\mathcal{M}^{k+\ell+1}(E^{k+\ell},\Q (m+j+1))@>>> H_\mathcal{M}^{k+\ell}(Z^{k+\ell},\Q (m+j))\\
@VV\hat{i}_* V @VV{i'_*} V @VV{i_{\infty, *}}V \\
H_\mathcal{M}^{k+\ell +3}(\hat{E}^{k,\ell ,*},\Q (m+j+2))@>>> H_\mathcal{M}^{k+\ell +3}(U^{k,\ell},\Q (m+j+2))@>>> H_\mathcal{M}^{k+\ell}(Z^{k,\ell},\Q (m+j)).\\
\end{CD}
$$
Consider the subgroup $G' = \mu_2^k\rtimes \mathfrak{S}_k$ of $G$. Let $\varepsilon'_k$ be the restriction of $\varepsilon_k$ to $G'$, and let $e'_k$ be the idempotent corresponding to $\varepsilon'_k$.

Denote $\tilde{\Xi}^{k,\ell,j}(\beta)=i'_* \circ \Delta_* \circ p^* (\mathrm{Eis}^{k+\ell-2j}(\beta))$. Consider the image of $\tilde{\Xi}^{k,\ell ,j}(\beta)$ under the residue map
$$
% H^{k+\ell+3}_{\mathcal{M}}
%(\hat{E}^k\times \hat{E}^{\ell},\mathbb{Q}(k+\ell-j+2))^{(\varepsilon_k,\varepsilon_\ell)}
%\longrightarrow  
\mathrm{Res}^{k,\ell ,j}:
H^{k+\ell+3}_{\mathcal{M}}(U^{k,\ell},\mathbb{Q}(k+\ell-j+2))^{(e'_k,e'_\ell)}
\to
%\overset{\mathrm{Res}^{k,\ell,j}}{\longrightarrow}
H^{k+\ell}_{\mathcal{M}}(Z^{k,\ell},\mathbb{Q}(k+\ell-j))^{(e'_k,e'_\ell)}
$$
%\begin{align*}
% H^{k+\ell+3}_{\mathcal{M}}&
%(\hat{E}^k\times \hat{E}^{\ell},\mathbb{Q}(k+\ell-j+2))^{(\varepsilon_k,\varepsilon_\ell)}\\
%&\longrightarrow  H^{k+\ell+3}_{\mathcal{M}}(U^{k,\ell},\mathbb{Q}(k+\ell-j+2))^{(\varepsilon_k,\varepsilon_\ell)}\\
%&\overset{\mathrm{Res}^{k,\ell,j}}{\longrightarrow}
%H^{k+\ell}_{\mathcal{M}}(Z^k\times Z^{\ell},\mathbb{Q}(k+\ell-j))^{(\varepsilon_k,\varepsilon_\ell)}
%\end{align*}
Note that $H^{k+\ell}_{\mathcal{M}}(Z^{k,\ell},\mathbb{Q}(k+\ell))^{(e'_k,e'_\ell)}
$ can be identified with
$H_{\mathcal{M}}^k(Z^k,\Q (k))^{e'_k}\otimes_\Q H_{\mathcal{M}}^{\ell}(Z^{\ell},\Q (\ell))^{e'_{\ell}}
\simeq\mathcal{F}^k_N \otimes_\Q \mathcal{F}^\ell_N$.
\begin{prop}
\begin{enumerate}
\item If $j> 0$, then we have $\mathrm{Res}^{k,\ell,j}\circ \tilde{\Xi}^{k,\ell,j}=0$.
\item If $j=0$, then $\mathrm{Res}^{k,\ell,0}\circ \tilde{\Xi}^{k,\ell,0}(\beta)$ is a nonzero multiple of
$\omega^{k+\ell}_N(\beta)\otimes  \omega^{k+\ell}_N(\beta)$.
%\in \mathcal{F}^{k}_N \otimes_\Q \mathcal{F}^{\ell}_N$.
\end{enumerate}
\end{prop}
\begin{proof}
\begin{enumerate}
\item The image of Eisenstein symbol is contained in
$H_{\mathcal{M}}^{k+\ell -2j} (Z^{k+\ell-j},\Q (k+\ell-2j))^{e_{k+\ell-2j}}$. Let $\Delta : \mathbb{G}_m^{k+\ell -j} \to \mathbb{G}_m^{k+\ell}$ be the diagonal embedding. We have
$$
\Delta_* : H_{\mathcal{M}}^{k+\ell -2j} (\mathbb{G}_m^{k+\ell-j},\Q (k+\ell-2j))^{e_{k+\ell-2j}'}\to 
H_{\mathcal{M}}^{k+\ell}(\mathbb{G}_m^{k+\ell},\Q (k+\ell-j))^{e_{k+\ell-2j}'}.
$$
By \cite[1.3.1 Lemma]{Scholl}, one has
%$$
%H_{\mathcal{M}}^{k+\ell -2j} (\mathbb{G}_m^{k+\ell-j},\Q (k+\ell-2j))^{e_{k+\ell-2j}}
%\simeq H_{\mathcal{M}}^0(\mathbb{G}_m^j,\Q (0))\simeq CH^0(\mathbb{G}_m^j)
%$$
%and
$$
H_{\mathcal{M}}^{k+\ell} (\mathbb{G}_m^{k+\ell},\Q (k+\ell-j))^{e_{k+\ell-2j}'}
\simeq H_{\mathcal{M}}^{2j}(\mathbb{G}_m^{2j},\Q (j))\simeq \mathrm{CH}^{j}(\mathbb{G}_m^{2j}).
$$
Since $j>0$, we have $\mathrm{CH}^{j}(\mathbb{G}_m^{2j})=0$. Therefore $\Delta_*=0$.
%Moreover it is easy to see that the map $\Delta_*:CH^0(\mathbb{G}_m^j) \to CH^j(\mathbb{G}_m^{2j})$
%is induced by the diagonal embedding $\mathbb{G}_m^j \to \mathbb{G}_m^{2j}$.
%Since the Chow group $CH^0(\mathbb{G}_m^j)$ is generated by the cycle $\mathbb{G}_m^j$,
%it is enough to see that the class of $\Delta_*(\mathbb{G}_m^j)$ is zero.
%This claim follows from the fact that the family of cycle
%$$Z_t =\{ (a_1,\ldots ,a_j,b_1,\ldots ,b_j,t) \in \mathbb{G}_m^{2j} \times \mathbb{P}^1 \, | \, a_i-b_i =t \textup{ for any }i \}$$
%satisfies $Z_0= \Delta_*(\mathbb{G}_m^j)$ and $Z_\infty=\emptyset$.
\item This follows from the commutativity of the diagram.
\end{enumerate}
\end{proof}
The closed embedding
$$i_{\mathrm{cusp}}:Z^k \times E^{\ell} \hookrightarrow U^{k,\ell}$$
%and
%$$i^2_{\mathrm{cusp}}:E^k \times Z^{\ell} \hookrightarrow U^{k,\ell}.$$
%These embeddings induce
induces
$$H^{k+\ell+1}_{\mathcal{M}}(Z^k \times E^{\ell},\mathbb{Q}(k+\ell+1))
\overset{i_{\mathrm{cusp},*}}{\longrightarrow} 
H^{k+\ell+3}_{\mathcal{M}}(U^{k,\ell},\mathbb{Q}(k+\ell+2)).$$
%and
%$$H^{k+\ell+1}_{\mathcal{M}}(E^k \times Z^{\ell},\mathbb{Q}(k+\ell+1))
%\overset{i^2_{\mathrm{cusp},*}}{\longrightarrow} 
%H^{k+\ell+3}_{\mathcal{M}}(U^{k,\ell},\mathbb{Q}(k+\ell+2)).$$
Let us consider Gysin morphisms
$$
%M_{gm}({E}^{\ell})\longrightarrow
%M_{gm}(\hat{E}^{\ell,*})\longrightarrow
\partial :
M_{gm}(Z^{\ell})(1)[2] \to %\overset{\partial}{\longrightarrow}
M_{gm}(E^{\ell})[1]
$$
for the pair $(\hat{E}^{\ell,*},Z^{\ell})$ and
$$
%M_{gm}(Z^k\times E^{\ell})\longrightarrow
%M_{gm}(Z^k\times \hat{E}^{\ell,*})\longrightarrow
\partial' :
M_{gm}(Z^k \times Z^{\ell})(1)[2]\simeq M_{gm}(Z^k)\otimes M_{gm}(Z^{\ell})(1)[2]\to
%\overset{\partial'}{\longrightarrow}
M_{gm}(Z^k \times E^{\ell})[1]=M_{gm}(Z^k)\otimes M_{gm}(E^{\ell})[1]
$$
for the pair $(Z^k\times \hat{E}^{\ell,*},Z^k\times Z^{\ell})$.
By \cite[Lemma 4.12]{Deglise}, it follows that $\partial'=1_{Z^k,*} \otimes \partial$.
Therefore we have the following commutative diagram:
$$
\begin{CD}
H_{\mathcal{M}}^k(Z^k,\Q (k))^{e'_k}\otimes_\Q H_{\mathcal{M}}^{\ell +1}(E^{\ell},\Q (\ell +1))^{e'_\ell}
@>1_{Z^k,*}\otimes \partial >> H_{\mathcal{M}}^k(Z^k,\Q (k))^{e'_k}\otimes_\Q H_{\mathcal{M}}^{\ell}(Z^{\ell},\Q (\ell ))^{e'_\ell}\\
@VV\mu V @VV\simeq V \\
H_{\mathcal{M}}^{k+\ell +1}(Z^k \times E^{\ell},\Q (k+\ell +1))^{(e'_k,e'_\ell )}@>\partial' >>
H_{\mathcal{M}}^{k+\ell}(Z^{k,\ell},\Q (k+\ell))^{(e'_k,e'_\ell)}\\
@VVi_{\mathrm{cusp},*}V @VV= V \\
H_{\mathcal{M}}^{k+\ell +3}(U^{k,\ell},\Q (k+\ell +2))^{(e'_k,e'_\ell )} @>\mathrm{Res}^{k,\ell ,0} >>
H_{\mathcal{M}}^{k+\ell}(Z^{k,\ell},\Q (k+\ell ))^{(e'_k,e'_\ell)},
\end{CD}
$$
where $\mu : H_{\mathcal{M}}^k(Z^k,\Q (k))^{e'_k}\otimes_\Q H_{\mathcal{M}}^{\ell +1}(E^{\ell},\Q (\ell +1))^{e'_\ell}
\to H_{\mathcal{M}}^{k+\ell +1}(Z^k \times E^{\ell},\Q (k+\ell +1))^{(e'_k,e'_\ell )}$ is the exterior product.

Since $\partial$ is surjective, $1_{Z^k,*}\otimes \partial$ is also surjective.
Hence $\mathrm{Res}^{k,\ell,0}\circ i_{\mathrm{cusp},*}$ is surjective.
It follows that there exists an element
$\xi_\beta \in H_{\mathcal{M}}^{k+\ell +1}(Z^k \times E^{\ell},\Q (k+\ell +1))^{(e'_k,e'_\ell )}$
such that $\mathrm{Res}^{k,\ell,0}\circ i_{\mathrm{cusp},*}(\xi_\beta )= \mathrm{Res}^{k,\ell,0}(\tilde{\Xi}^{k,\ell,0}(\beta))$.
%For a cusp $c\in X^{\infty}$, we have $\hat{\pi}_k^{-1}(c)\simeq \mathbb{G}_m^k$.
%%\item
%Then we have a canonical generator $\eta_c \in H^k_{\mathcal{M}}(\hat{\pi}_k^{-1}(c),\mathbb{Q}(k))\simeq \mathbb{Q}$.
%%\item
%Also we have the Eisenstein symbol $\mathrm{Eis}^{\ell}(\beta) \in H_{\mathcal{M}}^{\ell+1}(E^{\ell},\mathbb{Q}(\ell+1))$ for each $\beta \in \Q [\Z\slash N\Z]^0$.
%%\end{itemize}
%Then the exteior product
%$$\eta_c \times \mathrm{Eis}^{\ell}(\beta) \in H_{\mathcal{M}}^{k+\ell+1}(Z^k\times E^\ell,\mathbb{Q}(k+\ell+1))$$
%satisfies the formula
%$$
%\mathrm{Res}^{k,\ell,0}(i^1_{\mathrm{cusp},*}(\eta_c \times \mathrm{Eis}^{\ell}(\beta)))=
%\{ c \} \otimes C_\ell\cdot \omega^{\ell}_N(\beta)
%$$
%by \cite[Lemma 4.12]{Deglise}.
%Since $\omega^{\ell}_N$ is surjective, there exists a linear combination $\xi_\beta$ of
%$i^1_{\mathrm{cusp},*}(\eta_{c_1}\cup \mathrm{Eis}^{\ell}(\beta_1))$ and
%$i^2_{\mathrm{cusp},*}(\mathrm{Eis}^k(\beta_2)\cup \eta_{c_2})$
%such that
%$$
%\mathrm{Res}^{k,\ell,0}(\xi_\beta)=C_{k+\ell}^2\cdot \omega^{k+\ell}_N(\beta)\otimes \omega^{k+\ell}_N(\beta).
%$$
Now we define the generalized Beilinson-Flach element by
\[
\mathrm{BF}^{k,\ell,j}(\beta):=
\begin{cases}
\tilde{\Xi}^{k,\ell,j}(\beta) & \textup{ if $j > 0$},\\
\tilde{\Xi}^{k,\ell ,0}(\beta)-i_{\mathrm{cusp},*}(\xi_\beta) & \textup{ if $j=0$}.
\end{cases}
\]
%\begin{prop}
%\begin{align*}
From the definition, it is clear that
$$
\mathrm{BF}^{k,\ell,j}(\beta) \in  \mathrm{Im}\,% &
[H^{k+\ell+3}_{\mathcal{M}}
(\hat{E}^{k,\ell,*},\mathbb{Q}(k+\ell-j+2))^{(e'_k,e'_\ell)}%\\
%&
\to  H^{k+\ell+3}_{\mathcal{M}}(U^{k,\ell},\mathbb{Q}(k+\ell-j+2))^{(e'_k,e'_\ell)}].
$$
Note that the map $H^{k+\ell+3}_{\mathcal{M}}(\hat{E}^{k,\ell,*},\mathbb{Q}(k+\ell-j+2))^{(e'_k,e'_\ell)}
\to  H^{k+\ell+3}_{\mathcal{M}}(U^{k,\ell},\mathbb{Q}(k+\ell-j+2))^{(e'_k,e'_\ell)}$ is injective.

\section{Extensions to the boundary}\label{ext boundary}
To extend the motivic element $\Xi^{k,\ell,j}(\beta)$ to the boundary of the Kuga-Sato varieties, we use the following proposition.
\begin{prop}\label{prop boundary}
We have an isomorphism
$$H^{k+\ell+3}_{\mathcal{M}}
(\hat{E}^{k,*} \times \hat{E}^{\ell,*},\mathbb{Q}(k+\ell-j+2))^{(e'_k,e'_\ell)}\simeq
H^{k+\ell+3}_{\mathcal{M}}
(\overline{\overline{E}}^k\times \overline{\overline{E}}^{\ell},\mathbb{Q}(k+\ell-j+2))^{(e_k,e_\ell)}.$$
\end{prop}
To show the proposition, we prepare the following lemma.
\begin{lem}\label{lem boundary}
$M_{gm}(\hat{E}^{k,*})^{e'_k}\simeq M_{gm}(\overline{\overline{E}}^k)^{e_k}$
in $DM_{gm}^{\mathrm{eff}}(\Q)_{\Q}$.
\end{lem}

\begin{proof}[Proof of Lemma \ref{lem boundary}]
Let  $\overline{\overline{E}}^{k,\infty}$ be the complement of the smooth scheme $E^{k}$ in the smooth proper scheme  $\overline{\overline{E}}^{k}$.
Let $\overline{\overline{E}}^{k,\infty,\mathrm{reg}}$ be the intersection of $\overline{\overline{E}}^{k,\infty}$
with the non-singular part $\overline{E}^{k,\mathrm{reg}}$ of $\overline{E}^k$
and $\overline{\overline{E}}^{k,\infty,0} \subset \overline{\overline{E}}^{k,\infty,\mathrm{reg}}$
the intersection of $\overline{\overline{E}}^{k,\infty,\mathrm{reg}}$ with $\hat{E}^{k,*}$.
Note that the morphism
$\overline{\overline{E}}^{k} \to \overline{E}^{k}$ is an isomorphism over $\overline{E}^{k,\mathrm{reg}}$
by \cite[Theorem 3.1.0 (ii)]{Scholl},
hence $\overline{E}^{k,\mathrm{reg}}$ can be identified with a subscheme of
$\overline{\overline{E}}^k$ and
the open immersion $\overline{E}^{k,\mathrm{reg}}\hookrightarrow \overline{\overline{E}}^k$
induces an isomorphism
$$M_{gm}^c(\overline{\overline{E}}^k)^{e_k}\overset{\sim}{\longrightarrow}M_{gm}^c({\overline{E}}^{k,\mathrm{reg}})^{e_k}%=M_{gm}^c(\overline{\overline{E}}^{k,\mathrm{reg}})^{e_k}
$$
by \cite[Remark 3.8 (a)]{Wildeshaus}.
Also the connected component
$\hat{E}^{k,*}$ is identified with a subscheme of $\overline{E}^{k,\mathrm{reg}}$
by \cite[Theorem 3.1.0 (iii)]{Scholl}.
%Moreover we have
%$\overline{\overline{E}}^{k,\infty,0}\simeq \mathbb{G}_m^k\times X^\infty$
%canonically up to an automorphism $(x_1,\ldots,x_k)\mapsto (x_1,\ldots,x_k)^{-1}$
%and $\overline{\overline{E}}^{k,\infty,\mathrm{reg}}\simeq (\mathbb{G}_m \times \Z\slash N\Z )^k\times X^\infty$.
From these facts and \cite[Proof of Theorem 3.3]{Wildeshaus}, one has
$$
M_{gm}^c(\overline{\overline{E}}^{k,\infty})^{e_k}\overset{\sim}{\longrightarrow}
M_{gm}^c(\overline{\overline{E}}^{k,\infty,\mathrm{reg}})^{e_k}\overset{\sim}{\longrightarrow}
M_{gm}^c(\overline{\overline{E}}^{k,\infty,0})^{e_k'}.
$$

By \cite[Proposition 4.1.5]{Voevodsky} we have the distinguished triangles:
\[
\begin{CD}
M_{gm}^c(\overline{\overline{E}}^k)^{e_k}@>>> M_{gm}^c(E^k)^{e_k}@>>> M_{gm}^c(\overline{\overline{E}}^{k,\infty})^{e_k}[1]@>+1>> \\
@VV\simeq V  @VV=V @VV\simeq V \\
M_{gm}^c({\overline{E}}^{k,\mathrm{reg}})^{e_k}@>>> M_{gm}^c(E^k)^{e_k}@>>> M_{gm}^c(\overline{\overline{E}}^{k,\infty,\mathrm{reg}})^{e_k}[1]@>+1>> \\
@VVV  @VVV @VV\simeq V \\
M_{gm}^c(\hat{E}^{k,*})^{e_k'}@>>> M_{gm}^c(E^k)^{e_k'}@>>> M_{gm}^c(\overline{\overline{E}}^{k,\infty,0})^{e'_k}[1]@>+1>> \\
\end{CD}
\]
Moreover one has
$$
M_{gm}^c(E^k)^{e_k}\overset{\simeq}{\longrightarrow}M_{gm}^c(E^k)^{e_k'},
$$
since we have a decomposition $E^k =\coprod_{0\leq q \leq k}{\mathring{Y}}_q^k$ of $E^k$ into
locally closed subsets which are invariant under the action
of $\mathfrak{S}_{k+1}\cdot (\Z \slash N \Z )^{2k}$ as in \cite[Proof of 4.2 Theorem]{Schappacher-Scholl},
where
$$
\mathring{Y}_q^k = \{ (x_1,\ldots ,x_k) \in E^k \, \mid \,  \textup{exactly $q$ of the $x_i$'s are in $E[N]$} \} .
$$
%(use the fact that $E^k$ decomposes into locally closed subsets which are invariant under the action of
%$\mathfrak{S}_{k+1}\cdot (\Z \slash N \Z )^{2k}$).
From this fact, it follows that the inclusion $\hat{E}^{k,*} \hookrightarrow {\overline{E}}^{k,\mathrm{reg}}$ induces
$$
M_{gm}^c({\overline{E}}^{k,\mathrm{reg}})^{e_k}\overset{\simeq}{\longrightarrow}M_{gm}^c(\hat{E}^{k,*})^{e_k'}
$$
and hence
$$
M_{gm}^c(\overline{\overline{E}}^{k})^{e_k}\overset{\simeq}{\longrightarrow}M_{gm}^c(\hat{E}^{k,*})^{e_k'}.
$$
By duality for smooth schemes \cite[Theorem 4.3.7 3]{Voevodsky}, we have
$$
M_{gm}(\overline{\overline{E}}^k)^{e_k}\overset{\simeq}{\longleftarrow} M_{gm}({\overline{E}}^{k,\mathrm{reg}})^{e_k}\overset{\simeq}{\longleftarrow}M_{gm}(\hat{E}^{k,*})^{e_k'}.
$$
\end{proof}

\begin{proof}[Proof of Proposition \ref{prop boundary}]
Applying K\"unneth formula \cite[Proposition 4.1.7]{Voevodsky}, we have
$$M_{gm}(\hat{E}^{k,*} \times\hat{E}^{\ell,*})^{(e_k',e_{\ell}')}\simeq
M_{gm}(\overline{\overline{E}}^k\times \overline{\overline{E}}^{\ell})^{(e_k,e_{\ell})}.$$
By Voevodsky's definition of motivic cohomology,
%$$
%H_{\mathcal{M}}^i(X,\Q(j))=\mathrm{Hom}_{DM_{gm}^{\mathrm{eff}}(\Q)}(M_{gm}(X),\Q (j)[i]).
%$$
we have
$$H^{i}_{\mathcal{M}}
(\hat{E}^{k,*}\times \hat{E}^{\ell,*},\mathbb{Q}(j))^{(e_k',e_{\ell}')}\simeq H^{i}_{\mathcal{M}}
(\overline{\overline{E}}^k\times \overline{\overline{E}}^{\ell},\mathbb{Q}(j))^{(e_k,e_\ell)}$$
for any $i,j$. This completes the proof.
\end{proof}
Via the isomorphism in Proposition 8.1, we can consider $\mathrm{BF}^{k,\ell,j}(\beta )$
as an element of $H^{k+\ell+3}_{\mathcal{M}}(\overline{\overline{E}}^k\times \overline{\overline{E}}^{\ell},\mathbb{Q}(k+\ell-j+2))^{(e_k,e_\ell)}$.
\begin{prop}\label{BF}
We have $\langle r_{\mathcal{D}}(\mathrm{BF}^{k,\ell,j}(\beta)),\Omega_{f,g} \rangle
=\langle r_{\mathcal{D}}(\Xi^{k,\ell,j}(\beta)), \Omega_{f,g} \rangle$.
\end{prop}
\begin{proof}
The regulator map is compatible with the contravariance of morphisms by \cite[8.1.a)]{Jannsen-book}.
Therefore it is enough to show that
$\langle r_{\mathcal{D}}(i_{\mathrm{cusp},*}(\xi_\beta)),\Omega_{f,g} \rangle =0$.

By Jannsen's formula for the regulator in \cite[page 45]{Jannsen-survey},
the image of the regulator map is represented by an integral along $Z^k \times E^{\ell}$. Since $f^*$ is a cusp form, the differential form $\omega_{f^*}$ vanishes on the cycle $Z^k$.
Hence the differential form $\Omega_{f,g}$ vanishes on the cycle $Z^k \times E^{\ell}$.
Therefore the regulator integral vanishes.
\end{proof}

\section{Application to Beilinson's conjecture}\label{application}

Consider the projection to the $f \otimes g$-component
$$
\mathrm{pr}_{f,g}:H^{k+\ell+3}_{\mathcal{D}}
(\overline{\overline{E}}^k_\R \times \overline{\overline{E}}^{\ell}_\R,\mathbb{Q}(k+\ell+2-j))^{(e_k,e_\ell)}
\to H^{k+\ell +3}_{\mathcal{D}}(M(f \otimes g),\mathbb{R}(k+\ell+2-j)).
$$

Our results admit the following consequence for Beilinson's conjecture for the motive $M(f \otimes g)(k+\ell+2-j)$.

\begin{thm}\label{main thm}Let $f \in S_{k+2}(\Gamma_1(N_f),\chi_f)$ and $g \in S_{\ell+2}(\Gamma_1(N_g),\chi_g)$ be newforms with $k,\ell \geq 0$. Let $N$ be an integer divisible by $N_f$ and $N_g$, and let $j$ be an integer satisfying $0 \leq j \leq \min(k,\ell)$. Assume $\chi_f \chi_g \neq 1$ and $R_{f,g,N}(j+1)\neq 0$. Then there is an element
$\alpha \in H^{k+\ell+3}_{\mathcal{M}}
(\overline{\overline{E}}^k\times \overline{\overline{E}}^{\ell},\mathbb{Q}(k+\ell+2-j))^{(e_k,e_\ell)}$
such that
$$
\mathrm{pr}_{f,g} \circ r_{\mathcal{D}}(\alpha) = L^*(M(f\otimes g)(k+\ell+2-j)^{\vee}(1),0)\cdot t \mod K_{f,g}^{\times},
$$
where $t$ is a generator of the $K_{f,g}$-rational structure in $H^{k+\ell +3}_{\mathcal{D}}(M(f \otimes g),\mathbb{R}(k+\ell+2-j))$.
\end{thm}
\begin{proof}
Note that $R_{f,g,N}(j+1)\neq 0$ is equivalent to $R_{f^*,g^*,N}(j+1)\neq 0$. The theorem follows from Lemma \ref{Poincare}, Proposition \ref{regulator}, Proposition \ref{BF} and the fact that $\Omega_{f,g}$ is a $K_{f,g}^\times$-rational multiple of $\Omega$.
\end{proof}

Using the compatibility of Beilinson's conjecture with respect to the functional equation \cite[(2.2.2)]{Nekovar}, we get the following corollary.

\begin{cor}
Under the assumptions of Theorem \ref{main thm}, the weak version of Beilinson's conjecture for $L(f \otimes g,k+\ell+2-j)$ holds.
\end{cor}

\begin{rem}
\begin{enumerate}
\item The factor $R_{f,g,N}(j+1)$ is a product of local terms $R_{f,g,p}(j+1)$, where $p$ runs through the prime factors of $N$. If $p$ divides exactly one of the integers $N_f$ and $N_g$, then $R_{f,g,p}(s)=1$ by \cite[Theorem 15.1]{Jacquet}. If $p$ divides $N$ but doesn't divide $N_f N_g$, then $R_{f,g,p}(s)=1-\chi_f(p)\chi_g(p) p^{k+\ell+2-2s}$ by \cite[Lemma 1]{Shimura} and it may happen that $R_{f,g,p}(j+1)=0$, for example if $j=k=\ell$ and $p=1 \mod{\operatorname{lcm}(N_f,N_g)}$. Therefore, it is best to choose $N=\operatorname{lcm}(N_f,N_g)$ in Theorem \ref{main thm}. Moreover, it is easy to see that $R_{f,g,N}(j+1) \neq 0$ if $k+\ell -2j \not\in \{ 0,1,2 \}$ by \cite[Theorem 15.1]{Jacquet}.
\item The assumption $R_{f,g,N}(j+1)\neq 0$ is necessary.
We give an example.
There is a newform $f$ of weight $8$, level $39$ with character $\left( \frac{13}{\cdot}\right)$
such that $a_3(f)=-27$ (it is called 39.8.5a in the modular forms database \url{http://www.lmfdb.org/}).
Also there is a newform $g$ of weight 8, level 3 with trivial character
such that $a_3(g)=-27$ (it is called 3.8.a in the modular forms database).
Let $\pi_f=\bigotimes'_v \pi_{f,v}$ and $\pi_{g}=\bigotimes'_v \pi_{g,v}$
be the automorphic representations generated by $f$ and $g$.
Then it is easy to see that $\pi_{f,3}$ and $\pi_{g,3}$ are special representations
of the form $\mathrm{sp}(\sigma_{f,3}|\, \, |^{-\frac{1}{2}}, \sigma_{f,3} |\, \, |^{\frac{1}{2}})$
and $\mathrm{sp}(\sigma_{g,3}|\, \, |^{-\frac{1}{2}}, \sigma_{g,3} |\, \, |^{\frac{1}{2}})$,
where $\sigma_{f,3}$ and $\sigma_{g,3}$ are unramified characters of $\mathbb{Q}_3^{\times}$
satisfying $\sigma_{f,3}(3)=\sigma_{g,3}(3)=-1$.
By \cite[Theorem 15.1]{Jacquet}, we have
$$
L(\pi_{f,3}\otimes \pi_{g,3}, s)= L(\sigma_{f,3} \sigma_{g,3},s) L(\sigma_{f,3} \sigma_{g,3},s+1)
=(1-3^{-s})(1-3^{-s-1}).
$$
Hence the Euler factor of $L(f\otimes g,s)=L(\pi_f \otimes \pi_g, s-7)$ at $3$ is given by
$(1-a_3(f)a_3(g)3^{-s+1})(1-a_3(f)a_3(g)3^{-s})=(1-3^6\cdot 3^{-s+1})(1-3^6\cdot 3^{-s})$.
On the other hand, the Euler factor of $D(f,g,s)$ at $3$ is
$1-a_3(f)a_3(g)3^{-s}=1-3^6 \cdot 3^{-s}$.
Therefore $R_{f,g,N}(s)=1-3^6 \cdot 3^{-s+1}=1-3^7 \cdot 3^{-s}$.
If $j=6$, then $R_{f,g,39}(j+1)=0$.
\end{enumerate} 

\end{rem}

\end{document}